\documentclass[10pt]{amsart}
\usepackage[a4paper, bottom = 2.3cm, top= 2.5cm, left= 2.1 cm, right= 2.1 cm]{geometry}
\usepackage[T1]{fontenc}
\usepackage{graphicx}
\usepackage{indentfirst,csquotes}

\usepackage{amssymb,amsthm,amsmath}
\usepackage{xcolor,paralist,hyperref,titlesec,fancyhdr,etoolbox}
\newtheorem{theorem}{Theorem}[]
\newtheorem{Definition}[theorem]{Definition}

\newtheorem{lemma}[theorem]{Lemma}
\newtheorem{Proposition}[theorem]{Proposition}

\newtheorem{remark}[theorem]{Remark}

\numberwithin{equation}{section}
\usepackage{xfrac}
\usepackage{cleveref}
\hypersetup{
    colorlinks=true,
    linkcolor=blue,
    filecolor=magenta,      
    urlcolor=cyan,
    citecolor=green!40!black,
    pdftitle={Chemotaxis and Reactions in Anomalous Diffusion Dynamics},
    pdfauthor={De Andrade, Crystianne L., and Kiselev, Alexander A.},
    pdfsubject={},
    pdfkeywords={Keller-Segel model; Biological reactions;   Chemotaxis;  Reaction enhancement; Anomalous diffusion; Fractional Laplacian; Lévy flight; Mixing.}
}
\usepackage[sorting=nyt,  
    autocite=inline, 
    labelnumber=true,
    style = numeric-comp, 
    natbib=true,		    
    backend=biber,		    
    sortcites=true,		    
    maxbibnames=99,
    maxcitenames=1,
    uniquename=init,
    uniquelist=false,
    giveninits=true,  
    useprefix=true,    
    isbn=false,
    url=false,
    abbreviate=true,
    doi=true]{biblatex}
\usepackage{changepage}
\usepackage{enumitem} 
\setlist[enumerate]{label=\arabic*.}

\RequirePackage{xspace}
\newcommand{\ie}{i.\,e.,\xspace}
\usepackage{multicol}
\usepackage{tabularx}
\usepackage{comment}

\titleformat{\section}{\normalfont\centering\MakeUppercase}{\thesection}{0.2cm}{}
\titleformat{\subsection}{\normalfont\centering\MakeUppercase}{\thesubsection}{0.2cm}{}

\hypersetup{ colorlinks=true, linkcolor=black, filecolor=black, urlcolor=black }

\begin{document}
\thispagestyle{empty}
\title{Chemotaxis and Reactions in Anomalous Diffusion Dynamics} 
\author[]{De Andrade, Crystianne L., and Kiselev, Alexander A.}
\date{\today}

{\let\thefootnote\relax
\footnotemark 
 }

\begin{abstract}
Chemotaxis and reactions are fundamental processes in biology, often intricately intertwined. Chemotaxis, in particular, can be crucial in maintaining and accelerating a reaction. In this work, we extend the investigation initiated by \citet{Kiselev-Biomixing} by examining the impact of chemotactic attraction on reproduction and other processes in the context of anomalous diffusion of gamete densities. 
For that, we consider a partial differential equation, with a single density function, that includes advection, chemotaxis, absorbing reaction, and diffusion, incorporating the fractional Laplacian $\Lambda^\alpha$.  
The inclusion of the fractional Laplacian is motivated by experimental evidence supporting the efficacy of anomalous diffusion models, particularly in scenarios with sparse targets. The fractional Laplacian accommodates the nonlocal nature of superdiffusion processes, providing a more accurate representation than traditional diffusion models. Our proposed model represents a step forward in refining mathematical descriptions of cellular behaviors influenced by chemotactic cues.
\end{abstract} 
\maketitle
\bigskip


$\,$

$\,$

\section{Introduction} \label{section-Introduction}
Chemotaxis is the process through which cells convert chemical signals into motion behavior \citep{ KS-Review-2020, book-base}. 
The mathematical modeling of chemotaxis in cellular systems traces back to the works of  \citet{Keller1970Initiation} and \citet{Patlak-1953}, with the Keller-Segel model standing out as the most extensively studied one. 
This is a system of partial differential equations  describing the chemically induced movement of cells with density $\rho=\rho (x,t)$ towards increasing concentrations of a chemical substance with density $c=c(x,t)$, \ie  $\rho$ moves along the gradient of the chemical:
\begin{equation}
    \label{eq-1.1-the-Keller-Segel-model}
    \left\{\begin{aligned}
    &\partial_{t} \rho = \Delta \rho-\chi\nabla \cdot \left(\rho \nabla c\right)  & x \in \mathbb{R}^d,  & \quad t>0, \\
    &\tau \partial_{t} c =  \kappa \Delta c+\zeta \rho-\gamma c  \qquad & x \in \mathbb{R}^d,  & \quad t>0, \\
    & \rho(t=0)=\rho_{0}, \quad  c(t=0)=c_{0}   \qquad & x \in \mathbb{R}^d. & \; 
    \end{aligned}\right. 
\end{equation}
This system incorporates, in the right-hand side of the first equation, the diffusion of the cells, representing a random component of motion, and a chemotactic flux of advective type,  where $\chi$ stands for the chemotactic sensitivity. 
The second equation accounts for the diffusion, production, and consumption (or degradation) of the chemical signal, known as the chemoattractant, as it is emitted by the cells, diffused in the environment, and degraded at a rate proportional to the local concentration. 
As the biological background suggests, the model ensures non-negative solutions for cell density and chemical concentration under non-negative initial conditions. 
Moreover, as this model does not incorporate birth or death processes, describing exclusively cell motion, the total cell population, represented by the $L^1$ norm of $\rho$, is conserved over time: $\|\rho_0 \|_{L^1}=\|\rho(\cdot,t) \|_{L^1}$ 
\citep{KS-Review-2020, Bournaveas-2010-one-dimensional-fractional, book-base}.

The parabolic-elliptic Keller-Segel model arises from the assumption that the production and diffusion of the chemical occur much faster than other time scales in the problem $\left(\kappa \sim \zeta \gg 1\right)$. Additionally, with the observational time scale being much shorter than the speed at which the chemical signal $c$ degrades, $\gamma$ can be set equal to zero.   Under these conditions, the second equation of the Keller-Segel system, \eqref{eq-1.1-the-Keller-Segel-model}, simplifies to the elliptic form: $-\Delta c(x, t)=\rho(x, t)$. 

This simplification results in a single equation, since the gradient of the chemical concentration  can be expressed as function of $\rho$ \citep{Biler-On-parabolic-elliptic-limit-doubly-parabolic,Kiselev-Biomixing,kiselev2020chemotaxis}. That is,
\begin{equation*}
    \nabla c(x, t)=\nabla\left(-\Delta\right)^{-1} \rho(x,t)=-\frac{\Gamma\left((d+2) / 2\right)}{\pi^{d / 2} d}\int_{\mathbb{R}^d} \frac{x-y}{|x-y|^{d}} \rho(y,t) \mathrm{~d} y.
\end{equation*}
Thus, the classical parabolic-elliptic Keller-Segel model takes the following form 
\begin{equation}
    \label{eq-1.2-parabolic-elliptic-Keller-Segel-model}
    \left\{\begin{aligned}
    &\partial_t \rho=\Delta \rho-\chi \nabla \cdot\left(\rho \nabla\left(-\Delta\right)^{-1} \rho\right), & x \in \mathbb{R}^d,  & \quad t>0, \\
    &\rho(x, 0)=\rho_0(x)  \qquad & x \in \mathbb{R}^d. & \; 
    \end{aligned}\right.
\end{equation}

Several variations of the Keller-Segel model have been proposed to explore chemotaxis in conjunction with other phenomena, such as biological reactions. 
In that context, \citet{Kiselev-Biomixing} investigated the role of chemotaxis in enhancing biological reactions, focusing on coral broadcast spawning, a fertilization strategy adopted by various benthic invertebrates (sea urchins, anemones, corals). In this process, males and females release sperm and egg gametes into the surrounding flow, and the chemotaxis appears as the eggs release a chemical that attracts sperm. 

The mathematical model analyzed by \citet{Kiselev-Biomixing} is a modification of  \eqref{eq-1.2-parabolic-elliptic-Keller-Segel-model}, including advection and absorbing reaction, in which the approximation to one equation was based on the assumption that the chemical diffusion is much faster than the gamete densities diffusion.
Another simplification in this model is the assumption that sperm and egg gametes have identical densities, leading to a single density function $\rho \geq 0$. With these assumptions, the model is represented by the following equation:
%
\begin{equation}
    \label{eq-1.3-reaction-Keller-Segel-model}
    \left\{\begin{aligned}
    &\partial_t \rho+u \cdot \nabla \rho=\Delta \rho-\chi \nabla \cdot \left(\rho \nabla\left(-\Delta\right)^{-1} \rho\right)-\epsilon \rho^q,  \quad x \in \mathbb{R}^d,  \quad t>0, \\
    &\rho(x, 0)=\rho_0(x),  \qquad x \in \mathbb{R}^d, \quad d \geq 2,
    \end{aligned}\right.
\end{equation}
where  $u=u(x, t)$ is a given vector field modeling the ambient ocean flow ($u$ is divergence-free, regular, and prescribed, independent of $\rho$), 
and the term $\left(-\epsilon \rho^q\right)$ models the reaction (fertilization), with 
the parameter $\epsilon$ regulating the strength of the fertilization process.   
They pointed out that the value of $\epsilon$ is small due to the fact that an egg gets fertilized only if a sperm attaches to a certain limited area on its surface. 
The model does not track the production of reaction -- fertilized eggs. 

In this work, we extend the investigation initiated in \cite{Kiselev-Biomixing} by examining the impact of chemotactic attraction on reproduction processes in the context of anomalous diffusion of gamete densities.  
To this purpose, we analyze the behavior of the total fraction of unfertilized eggs over time $t$:
\begin{equation}
    \label{unfertilized-eggs}
    m(t)\equiv\int_{\mathbb{R}^d} \rho(x, t) \mathrm{~d} x,
\end{equation}  
where $\rho$ is the solution to the modified single partial differential equation 
\begin{equation}
    \label{eq-4.1-k}
    \left\{\begin{array}{l}
    \partial_t \rho+u \cdot \nabla \rho=-\Lambda^\alpha \rho-\chi \nabla \cdot\left(\rho \nabla\left(-\Delta\right)^{-1} \rho\right)-\epsilon \rho^q, \quad \alpha \in (1,2] \\
    \rho(x, 0)=\rho_0(x), \quad x \in \mathbb{R}^d \quad d \geq 2.
    \end{array}\right.
\end{equation}
In this equation, $\Lambda^{\alpha}=\left(-\Delta\right)^{\alpha / 2}$ is the fractional Laplacian, which can be defined in Fourier variables as 
\begin{equation}
    \label{eq-definition-fractional-Laplacian-fourier}
    \widehat{(\Lambda^{\alpha} f)}(\xi)=\widehat{\left[\left(-\Delta\right)^{\alpha / 2}f\right]}(\xi)=|\xi|^\alpha \widehat{f}(\xi), \; \; \forall \xi \in \mathbb R^d, \text{and } \; \alpha \in(0,2],
\end{equation}
for a sufficiently regular function $f$, with the Fourier transform convention   
\begin{equation}
    \label{eq-definition-Fourier-transform}
    \mathcal{F} f(\xi)=\frac{1}{(2\pi)^{d/2}}\int_{\mathbb{R}^d} f(x) \mathrm{e}^{-\mathrm{i} x \cdot \xi} \mathrm{~d} x,
\end{equation}
where $x \cdot \xi$ is the inner product in $\mathbb{R}^d$. 
For $\alpha \in (0, 2)$, an alternative definition is given by 
\begin{equation}
    \label{eq-definition-fractional-Laplacian-integral}
    \Lambda^{\alpha} f(x)=c_{d, \alpha} P . V . \int _{\mathbb{R}^d} \frac{f(x)-f(y)}{|x-y|^{d+\alpha}} \mathrm{~d} y,
\end{equation}
where $c_{d, \alpha}=\frac{2^{\alpha-1} \alpha \Gamma\left(\frac{d+\alpha}{2}\right)}{\pi^{d / 2} \Gamma\left(1-\frac{\alpha}{2}\right)}$ is a normalization constant  and $P.V.$ stands for the Cauchy principal value.

Note that, in standard diffusion, the use of the Laplacian ($\Lambda^{2}=-\Delta$) and the first-order time derivative ($\partial_t$) results in a Gaussian process for the particle motion. At the microscopic level, this leads to the mean squared displacement (MSD)\footnote{The MSD is a measure of the deviation of the position of a particle concerning a reference position over time.} of a particle being a linear function of time: $\left\langle x^2 \right\rangle \sim t$. 
In contrast, anomalous diffusion is characterized by an MSD that increases nonlinearly; it grows either slower (subdiffusion) or faster (superdiffusion) for large times than that in a Gaussian process. At the macroscopic level, such processes are adequately described by fractional differential equations, which involve fractional derivatives in time and/or space. 
Specifically, superdiffusion processes are modeled using equations with fractional derivative in space. 
In that case, the standard Laplace operator is replaced by the fractional one, which is nonlocal, since it is an integrodifferential operator \citep{Fractional-MEERSCHAERT2006181, Guo-2015-Partial-Differential-Equations, inverse-problems-anomalous-diffusion-processes-2015}. 

In the chemotaxis phenomena, many experiments have shown that anomalous diffusion models (superdiffusion type of behavior) can offer a superior fit to experimental data, especially when chemoattractants, food, or other targets are sparse or rare in the environment. 
At a microscopic level, superdiffusion processes can be described by Lévy flights or Lévy walk, where the length of particle jumps follows some heavy-tailed distribution, reflecting the long-range interactions among particles, which means that large jumps dominate the more common smaller ones. This non-negligible probability for long positional jumps in a biological context translates into cells/organisms persisting in a single direction of motion for a substantially longer time than in typical random walks \citep{Escudero-2006-Fractional-model, Estrada-Rodriguez-2018, Guo-2015-Partial-Differential-Equations}.

Nonetheless, while chemotaxis is well-established within the framework of classical diffusion as a mechanism that enhances and sustains reaction processes (e.g., \cite{Kiselev-Biomixing}), its influence on anomalous diffusion remains largely unexplored. 
It is worth highlighting that anomalous diffusion processes characterized by a fractional diffusion exponent $\alpha < 2$ introduce significant challenges, such as nonlocal effects and weaker dissipative properties compared to classical diffusion $\alpha = 2$, which makes the system more susceptible to blow-ups.  
This research aims to bridge the gap in understanding how nonlocality and fractional dynamics influence the behavior of chemotactic systems. By extending classical models to include anomalous diffusion, it provides a novel framework to study the interplay between chemotaxis and anomalous transport, with particular emphasis on analyzing the key quantity $m(t)$, through which we can access the fraction that reacts (or the amount of egg gametes that have been successfully fertilized), $\displaystyle m(0)-\lim _{t \rightarrow \infty} m(t)$. 

\begin{remark} \label{monotone-decreasing-fraction-unfertilized-eggs}
    The total fraction of unfertilized eggs by time $t$, $m(t)$, given by \eqref{unfertilized-eggs}, is a monotone decreasing function. Indeed, consider that $\rho$ vanishes when $|x| \rightarrow \infty$, and $\widehat{(\Lambda^{\alpha} \rho)}(0)$ is well defined. Then, as from \eqref{eq-definition-fractional-Laplacian-fourier} we have 
    \begin{equation*}
        \widehat{(\Lambda^{\alpha} \rho)}(0)=\int_{\mathbb{R}^d}\Lambda^{\alpha} \rho \mathrm{~d}x=0,
    \end{equation*}
    a direct application of the boundary conditions yields    
    \begin{equation*}
        \frac{\mathrm{d}}{\mathrm{d} t}\|\rho(\cdot, t)\|_{L^{1}}=-\epsilon \int_{\mathbb{R}^d} \rho^q \mathrm{~d} x.
    \end{equation*}
\end{remark} 

We want to see how fertilization can be efficient in both chemotactic and chemotaxis-free scenarios.  For that, we prove the following results. \vspace{0.2cm}
\begin{theorem} \label{Theorem-1.1-k}
    Let $\rho$ solve \eqref{eq-4.1-k} with a divergence-free $u \in C^{\infty}(\mathbb{R}^d \times[0, \infty))$ and initial data $\rho_0 \geq 0 \in \mathcal{S}$ (Schwartz space).
    Assume that $q d>d+\alpha$ and the chemotaxis is absent: $\chi=0$. Then, there exists a constant $C_0$ depending only on $\epsilon$, $q$, $d$, $\alpha$, and $\rho_0$ but not on $u$ such that $m(t) \geq C_0$ for all $t \geq 0$.
    Moreover, $C_0 (\epsilon,q, \alpha, d, \rho_0) \rightarrow  m(0) \equiv m_0$ as $\epsilon \rightarrow 0$ while $\rho_0$, $u$ and $q$ are fixed.
\end{theorem}
\vspace{0.1cm}

However, when chemotaxis is present, we have

\vspace{0.1cm}
\begin{theorem} \label{Theorem-1.2-k}
    Let $\rho$ solve \eqref{eq-4.1-k} with a divergence-free $u \in C^{\infty}(\mathbb{R}^d \times[0, \infty))$ and initial data $\rho_0 \geq 0 \in \mathcal{S}$.
    Assume that $d=2$ and $q$ is a positive integer greater than $2$. 
    Then we have 
    \begin{equation*}
         C_1(\epsilon,q, \alpha,\chi, \rho_0) \leq \lim_{t \rightarrow \infty} m(t) < C_2(\alpha,\chi, \rho_0, u)
    \end{equation*}
    with $C_1,C_2>0$. Moreover, $C_1(\epsilon,q, \alpha,\chi, \rho_0) \rightarrow 0$ and $C_2(\alpha,\chi, \rho_0, u)\rightarrow 0$ as $\chi \rightarrow \infty$, with $q$, $\rho_0$ and $u$ fixed.
\end{theorem}
\begin{remark} 
\textcolor{white}{.} \par
\begin{adjustwidth}{-0.01cm}{0cm}
\begin{enumerate}
    \item By $u \in C^{\infty}(\mathbb{R}^d \times[0, \infty))$, we mean that, additionally to being infinitely differentiable on $\mathbb{R}^d\times[0, \infty)$, $u$ and all its derivatives are uniformly bounded over $(x, t) \in \mathbb{R}^d \times[0, T]$, for every $T>0$. 
    This definition is adjusted for our purposes and differs from the standard interpretation of a function on $C^{\infty}(\mathbb{R}^d \times[0, \infty))$. 
    Accordingly, the norm $\|u\|_{\mathcal{C}^s}$ is defined as 
    \begin{equation*}
         \|u\|_{\mathcal{C}^s} = \sum_{|\kappa| \leq s} \sup_{(x, t) \in \mathbb{R}^d \times [0, T]} \left| \partial^{\kappa} u(x, t) \right|,
    \end{equation*}
    where $\kappa$ is a multi-index, with this definition holding for every $T > 0$.
    \item The assumptions that $u \in C^{\infty}(\mathbb{R}^d \times[0, \infty))$ and $\rho_0 \in \mathcal{S}$ can be relaxed. It suffices for $\rho_0$, for example, to exhibit sufficiently rapid decay and minimal regularity.  
    \item \Cref{Theorem-1.2-k} comprises two parts: \Cref{Theorem-4.3-k}, where we establish a lower bound for the $L^1$ norm of $\rho$, and \Cref{Theorem-4.2-k}, where we establish an upper bound. Additionally, in these theorems, we relax the requirement that the initial data, $\rho_0$, belongs to $\mathcal{S}$.
    \item In the presence of chemotaxis, particularly for $d=2$, \citet{Kiselev-Biomixing} established, for $\alpha=2$, the existence of solutions where the lower bound of its $L^1$ norm is independent of the reaction term, and thus independent of the coupling of the reaction term $\epsilon$. 
    However, for  $1<\alpha <2$, the possible dependence on $\epsilon$ of a condition in the proof of \Cref{Theorem-4.2-k} (condition \eqref{eq-assumption}) challenges this assertion.
    Therefore, we can not affirm that the lower bound of the $L^1$ norm does not depend on $\epsilon$.
\end{enumerate}
\end{adjustwidth}
\end{remark}

Note that the constants $C_0$ and $C_1$ are independent of the flow field $u$. This suggests that the reaction rate has a limited threshold, above which it cannot be increased, irrespective of the flow's strength or form. 
Additionally, we can see, from the proof of \Cref{Theorem-1.1-k}, that  in chemotaxis-free settings, the quantity that reacts, $\displaystyle m_0-\lim _{t \rightarrow \infty} m(t)$, exhibits a decrease of order $\epsilon$.

In scenarios involving chemotaxis, the chemotactic sensitivity $\chi$ also contributes to shaping lower bound of the total fraction of unfertilized eggs.
Moreover, in these scenarios, the reaction term does not quantitatively impact the estimates for the upper bound for $m(t)$, that is, all estimates of the $L^1$ norm of $\rho$ are independent of the coupling of the reaction term, $\epsilon$. 
This implies that the amount of the density that reacts 
satisfies a lower bound independent of $\epsilon$.  
Note that the chemotactic term, in contrast to the flow and diffusion alone, plays a crucial role in achieving highly efficient fertilization rates, as in the chemotactic scenario, both the upper and lower bounds of $m(t)$ decrease as the chemotaxis strength, $\chi$, increases.
\begin{remark} \label{Remark-classical-Keller-Segel-X-reaction-term}
The reaction term itself may not directly impact the total fraction of reacting density, but its presence plays a crucial role in controlling the growth of the solution's $L^{\infty}$ norm. This regulation arises from the balance between chemotaxis and the reaction term.
For instance, in a scenario in two dimensions where the reaction term is absent and $\alpha=2$ (classical parabolic-elliptic Keller-Segel), for an initial mass sufficiently large, chemotaxis alone would lead to a blowup, causing the density $\rho$ to concentrate as a Dirac mass.
\end{remark}

Furthermore, we prove the existence of solution for any initial condition sufficiently regular, emphasizing that the presence of an additional negative reaction term $-\epsilon \rho^q$ with $q>2$ prevents the solution from losing regularity in finite time, as proved by \citet{Kiselev-Biomixing} for model \eqref{eq-1.3-reaction-Keller-Segel-model}. 
It is crucial to stress that, as highlighted in \Cref{Remark-classical-Keller-Segel-X-reaction-term}, the classical parabolic-elliptic Keller-Segel system \eqref{eq-1.2-parabolic-elliptic-Keller-Segel-model}, for $d\geq 2$, does not follow the same dynamics.
This system exhibits critical behavior in $L^{d/2}$, \ie a small initial condition in the $L^{d/2}$ space ensures global well-posedness over time, while a large one leads to a blowup, signifying that the solution does not remain bounded; 
in particular, for $d=2$, solutions exist globally in time for initial masses less than  $8 \pi/\chi$, while a blowup occurs in finite time for  $m_0>8 \pi/\chi$ \citep{Nagai-1995, Horstmann03from1970, book-base, Blanchet-article-parabolic-elliptic-PKS}. 

This critical behavior also holds for the nonlocal parabolic-elliptic  Keller-Segel system, represented by equation \eqref{eq-1.2-parabolic-elliptic-Keller-Segel-model} with fractional Laplacian $1<\alpha<2$, for which the existence of blowing-up solutions was proved \citep{Biler-fractional-9, Biler-fractional-2017, Exploding-solutions-nonlocal-quadratic-evolution}.
In this case, \citet{Biler-fractional-9} proved that the system exhibits critical behavior in $L^{d/\alpha}$ space.
\section{Framework Overview}  \label{Plan-of-the-work}
The structure of this work is designed to provide a systematic exploration of our research objectives. 
%
In \Cref{sec:Preliminaries}, we introduce the specific type of solution considered for \eqref{eq-4.1-k}, establish basic notation, and present key results that play a pivotal role in proving the existence of the solution. 

Proceeding to \Cref{sec:Global-Existence}, we ensure local existence and uniqueness of solutions using a standard fixed-point procedure. Subsequently, we demonstrate control over the growth of solution norms, culminating in the establishment of global well-posedness.  
%
Then, we end this section showing that for a nonnegative initial condition in \eqref{eq-4.1-k}, the solution remains nonnegative. 
It is crucial to emphasize that our focus on nonnegativity is particularly pertinent, given that $\rho$ in \eqref{eq-4.1-k} represents densities, which inherently requires $\rho$ to remain nonnegative.

Finally, in \Cref{sec:Reaction-efficiency}, we delve into an analysis of the behavior of the total fraction of unfertilized eggs, expressed by \eqref{unfertilized-eggs}, in both chemotaxis and chemotaxis-free scenarios.  
\section{Preliminaries}
\label{sec:Preliminaries}
\begin{remark}[\textbf{Notation}]
Throughout this work, constants, which may change from line to line,  will be denoted by the same letter $C$, and are independent of $\rho$, $x$, and $t$. The notation $C = C(*)$ emphasizes the dependence of $C$ on a parameter  “$*$”.
%
\end{remark} 
\begin{Definition}[\textbf{Mild solution}] 
Let us rewrite equation \eqref{eq-4.1-k} in an integral form using the Duhamel's principle: \vspace{-0.27cm}
\begin{equation}
    \label{eq-5.4-k}
    \rho(x, t)=K^{\alpha}_t *\rho_0(x)+\int_0^t K^{\alpha}_{t-r} *\left(-\nabla \cdot(u \rho)-\epsilon \rho^q+\chi \nabla \cdot\left(\rho \nabla \Delta^{-1} \rho\right)\right) \mathrm{d} r,
\end{equation}
where $K^{\alpha}_{t}(x)$ is defined through the Fourier transform as
\begin{equation}
    \label{definition-kt}
    \begin{split}
        K^{\alpha}_{t}(x) & =\int_{\mathbb{R}^{d}} \mathrm{e}^{-t(2\pi|\xi|)^{\alpha}}\mathrm{e}^{2\pi i x \cdot \xi}  \mathrm{d} \xi
        \\
        &=(2 \pi)^{-d} t^{-\frac{d}{\alpha}} \int_{\mathbb{R}^{d}} \mathrm{e}^{-|\xi|^{\alpha}}\mathrm{e}^{i \frac{x}{t^{1/ \alpha}} \cdot \xi} \mathrm{d} \xi \\ 
        & = t^{-\frac{d}{\alpha}} K^{\alpha}\left(\frac{x}{t^{1/\alpha}}\right).
    \end{split}
\end{equation}
Thus, the kernel function can be written as 
\begin{equation}
    \label{definition-k}
    K^{\alpha}(x)=(2 \pi)^{-d} \int_{\mathbb{R}^{d}} \mathrm{e}^{i x \cdot \xi} \mathrm{e}^{-|\xi|^{\alpha}} \mathrm{d} \xi.
\end{equation}

A solution to \eqref{eq-5.4-k} is called the mild solution to  system \eqref{eq-4.1-k}.
\end{Definition}

\begin{lemma} \label{integral-k}
     Let $\alpha \in (0,2]$, and consider the functions $K^{\alpha}_t $ and $K^{\alpha}$ in \eqref{definition-kt} and \eqref{definition-k}, respectively. Then,
     \renewcommand{\labelenumi}{(\alph{enumi})}
     \renewcommand{\theenumi}{(\alph{enumi})}
     \begin{enumerate}
        \item $K^{\alpha} \in  C_0^{\infty}$; \label{estimate-k(a)}
        \item For any $1 \leq p \leq \infty$, $K^{\alpha}$, $ \nabla K^{\alpha}$, $K^{\alpha}_t$, $ \nabla  K^{\alpha}_t \in L^p$, where $0<t<\infty$  \citep{Miao-2008-Well-posedness}; \label{estimate-k(b)}
        \item $\displaystyle \int_{\mathbb{R}^d} K^{\alpha}(x)   \mathrm{~d} x=1$, $\displaystyle \int_{\mathbb{R}^d} K^{\alpha}_{t}(x)   \mathrm{~d} x=1$, and $\displaystyle \int_{\mathbb{R}^d} \nabla K^{\alpha}(x)   \mathrm{~d} x=0.$ \label{estimate-k(c)}
     \end{enumerate}
\end{lemma}
\begin{proof}
    \ref{estimate-k(a)} Note that $\widehat{K^{\alpha}}\left(\xi\right)=\mathrm{e}^{-|2 \pi \xi|^{\alpha}} \in L^1$. Thus, from the Fourier transform, $K^{\alpha} \in L^{\infty} \cap C$. Moreover, from the Riemann Lebesgue Lemma, \, $ \mathcal{F} (L^1) \subset C_0$, $K^{\alpha} \in L^{\infty} \cap C_0$. 
    In an analogous way, we have  $\nabla K^{\alpha} \in L^{\infty} \cap C_0$,  since $2 \pi \mathrm{i} \xi \mathrm{e}^{-|2 \pi \xi|^{ \alpha}} \in (L^1)^d$. Additionally, $\partial^\eta K^{\alpha} \in L^{\infty} \cap C_0$,  since $ (2 \pi \xi)^\eta \mathrm{e}^{-|2 \pi \xi|^{ \alpha}} \in L^1$, where for higher-order derivatives we use multi-index notation.   
    Therefore, we conclude that $K^{\alpha} \in  C_0^{\infty}$.   \vspace{0.2cm}
    
    \ref{estimate-k(b)} See \citep{Miao-2008-Well-posedness}. \vspace{0.2cm}

    \ref{estimate-k(c)}
    From the Fourier Inversion Theorem, we know that if $f$ and its Fourier transform $\widehat{f}$ belong to $L^1$, then $f$ agrees almost everywhere with a continuous function $f_0$, and $(\widehat{f})^{\vee}=\left(f^{\vee}\right)^{\wedge}=f_0$. Applying this to our context where $K^{\alpha}$ and its Fourier transform $\widehat{K^{\alpha}}(\xi)$ are in $L^1$, we  deduce that $\displaystyle \int_{\mathbb{R}^d} K^{\alpha}(x) \mathrm{~d} x =\widehat{K^{\alpha}}(0)=1$.
    
    Furthermore, considering the transformation $T(x)=t^{-1} x$ for $t>0$, the expression $(f \circ T) \widehat{\; \;}(\xi)$ is given by $t^d \widehat{f}(t \xi)$. Applying this to our scenario where $K^{\alpha}_{t}(x) = t^{-\frac{d}{\alpha}} K^{\alpha}\left(\frac{x}{t^{1/\alpha}}\right)$, we obtain $ \displaystyle \int_{\mathbb{R}^d} K^{\alpha}_{t}(x) \mathrm{~d} x  =\widehat{K^{\alpha}_{t}}(0)=\widehat{K^{\alpha}}(0)=1$.
%
    Moreover, utilizing the property that for $f \in C^1$ with $\nabla f \in L^1$, the Fourier transform $\left(\nabla f\right) \widehat{\;\;}(\xi)$ equals $2 \pi i \xi \widehat{f}(\xi)$, we derive the final equality in \ref{estimate-k(c)}.
\end{proof}
\begin{Definition}[Banach space $K_{s, \beta}$]
    Let $d$ be an arbitrary dimension.  We define the Banach space $K_{s, \beta}$ by the norm $\|\varphi\|_{K_{s, \beta}}=\|\varphi\|_{M_{\beta}}+\|\varphi\|_{H^s}$, where $s \geq 0$, $\alpha \in (0,2]$, and $0 \leq \beta < \alpha$. Here, $H^{s}$ is the standard Sobolev spaces in $\mathbb{R}^{d}$, and the norm $\|\cdot\|_{M_{\beta}}$ is given by
    \begin{equation*}
        \|\varphi\|_{M_{\beta}}=\int_{\mathbb{R}^d}(|\varphi(x)|+|\nabla \varphi(x)|)\left(1+|x|^\beta\right) \mathrm{~d} x.
    \end{equation*}
\end{Definition}
 
Now, consider the following Lemmas: 
\begin{lemma}{\citep[Lemma 6]{inequality-2005}} \label{lemma-inequality-2005} 
    Let $\gamma, \; \beta$ be multi-indices, $|\beta|<|\gamma|+\alpha \max(j,1)$, $j = 0, \; 1, \; 2, \;  \cdots ,$ $1 \leq p \leq \infty$, and $\alpha \in (0,2]$. Then,
    \begin{equation}
        \label{estimativa-k-lemma-inequality-2005}
        \|x^{\beta} D_{t}^{j} D^{\gamma}K^{\alpha}_t \|_{L^{p}}= Ct^{\frac{|\beta|-|\gamma|}{\alpha}-j-\frac{d(p-1)}{\alpha p}} 
    \end{equation}
    for some constant C depending only on $\alpha, \; \gamma, \; \beta, \; j, \; p,$ and the space dimension $d$.
\end{lemma}

\begin{lemma} \label{Lemma-5.1-k}
    Assume $\varphi \in K_{s, \beta}$, with $s \geq 0$, $\alpha \in (0,2]$, and $0 \leq \beta < \alpha$ . We have 
    \begin{gather}
        \|K^{\alpha}_t *\varphi\|_{M_{\beta}} \leq C\left(1+t^{\beta / \alpha}\right)\|\varphi\|_{M_{\beta}}, \quad\|\nabla K^{\alpha}_t *\varphi\|_{M_{\beta}} \leq C\left(t^{-1 / \alpha}+t^{(\beta-1) / \alpha}\right)\|\varphi\|_{M_{\beta}} , \label{estimativa-1-k}\\
        \|K^{\alpha}_t *\varphi\|_{H^s} \leq\|\varphi\|_{H^s}, \quad\|\nabla K^{\alpha}_t *\varphi\|_{H^s} \leq C t^{-1 / \alpha}\|\varphi\|_{H^s}. \label{estimativa-2-k}
    \end{gather}   
    \indent  As a consequence, \vspace{-0.3cm}
    \begin{equation}
        \label{estimativa-3-k}
        \|K^{\alpha}_t *\varphi\|_{K_{s, \beta}} \leq C\left(1+t^{\beta / \alpha}\right)\|\varphi\|_{K_{s, \beta}}, \quad \text{ and } \quad \|\nabla K^{\alpha}_t *\varphi\|_{K_{s, \beta}} \leq C\left(t^{-1 / \alpha}+t^{(\beta-1) / \alpha}\right)\|\varphi\|_{K_{s, \beta}} .
    \end{equation}
\end{lemma}
\begin{proof}  
The proof of \eqref{estimativa-2-k} is elementary and we omit it. To establish \eqref{estimativa-1-k}, note that 
\begin{equation*}
    (1+|x|^{\beta})\leq C\left(1+|x-y|^{\beta}\right) \left(1+|y|^{\beta} \right) \qquad \text{and} \qquad \left(1+|x|^\beta\right) \leq 1+\sum_{i=1}^{d} |x_i|^\beta.
\end{equation*}
Hence, considering $\varphi \in M_{\beta}$ and $\psi, \, x_i^\beta  \psi \in L^{1}$ for $1\leq i \leq d$, employing Fubini's theorem yields  
\begin{equation*}
    \begin{split}
        \|\psi*\varphi\|_{M_{\beta}} & \leq \int_{\mathbb{R}^d}\int_{\mathbb{R}^d} \left(1+|x|^{\beta}\right)  |\psi(x-y)| \left( |\varphi(y)|+ |\nabla \varphi(y)|\right)\mathrm{~d} y \mathrm{~d} x\\ 
        & \leq C \int_{\mathbb{R}^d}\int_{\mathbb{R}^d} \left(1+|x-y|^{\beta}\right) \left(1+|y|^{\beta} \right)  |\psi(x-y)| \left(|\varphi(y)|+ |\nabla \varphi(y)|\right)\mathrm{~d} y \mathrm{~d} x\\ 
        & \leq C \|(1+|\cdot |^{\beta}) \psi \|_{L^{1}} \|\varphi\|_{M_{\beta}} \\
        & \leq  C \left\|\left(1+\sum_{i=1}^{d} |x_i|^\beta\right)  \psi \right\|_{L^{1}} \|\varphi\|_{M_{\beta}} \\
        & \leq  C\left( \|\psi \|_{L^{1}}+ \sum_{i=1}^{d} \| x_i^\beta \psi \|_{L^{1}} \right)\|\varphi\|_{M_{\beta}}. 
    \end{split}     
\end{equation*}

Now, set $\psi=K^{\alpha}_t$. From \eqref{estimativa-k-lemma-inequality-2005}, as $\beta <\alpha$, we obtain
\begin{equation*}
    \begin{split}
        \|K^{\alpha}_t *\varphi\|_{M_{\beta}} & \leq C\left( \|K^{\alpha}_t  \|_{L^{1}}+ \sum_{i=1}^{d} \| x_i^\beta K^{\alpha}_t  \|_{L^{1}} \right)\|\varphi\|_{M_{\beta}} \leq C (1+t^{\frac{\beta}{\alpha}}) \|\varphi\|_{M_{\beta}}.        
    \end{split}
\end{equation*}
Next, set $\psi=\nabla K^{\alpha}_t$. Again from   \eqref{estimativa-k-lemma-inequality-2005}, as $\beta <\alpha+1$, we find that
\begin{equation*}
    \begin{split}
        \|\nabla K^{\alpha}_t *\varphi\|_{M_{\beta}} & \leq C\left( \|\nabla K^{\alpha}_t  \|_{L^{1}}+ \sum_{i=1}^{d} \| x_i^\beta \nabla K^{\alpha}_t  \|_{L^{1}} \right)\|\varphi\|_{M_{\beta}} \leq C (t^{-\frac{1}{\alpha}}+t^{\frac{\beta-1}{\alpha}}) \|\varphi\|_{M_{\beta}}.        
    \end{split}
\end{equation*}

Finally, observe that \eqref{estimativa-3-k} is a direct consequence of the preceding calculations and the definition of the norm $\|\, \cdot \,\|_{K_{s, \beta}}$.
\end{proof}
\begin{remark}
    The constraint on the value of $\beta$ is rooted in the characteristic behavior of non-Gaussian Lévy $\alpha$-stable $(0<\alpha<2)$ semigroups. These semigroups exhibit decay solely at an algebraic rate of $|x|^{-d-\alpha}$ as $|x| \rightarrow \infty$ $\left(|K^{\alpha}(x)| \leq C(1+|x|)^{-d-\alpha}\right)$. Note that the divergence from this pattern in the case of the Gaussian kernel ($\alpha=2$), which undergoes exponential decay. 
\end{remark}

\begin{lemma}\label{Lemma-5.2-k}
    Let $q$ and $s$ be positive integers such that $s > d/2+ 1$, and $\beta \geq 0$. Consider $f, g \in K_{s, \beta}$. Then, the following estimates hold: \vspace{-0.3cm}
    \begin{gather}
        \|f^{q}-g^{q}\|_{H^{s}} \leq C\left(\|f\|_{H^{s}}^{q-1}+\|g\|_{H^{s}}^{q-1}\right)\|f-g\|_{H^{s}}, \label{eq-5.5-k}\\
        \|f^{q}-g^{q}\|_{M_{\beta}} \leq C\left(\|f\|_{H^{s}}^{q-1}+\|g\|_{H^{s}}^{q-1}\right)\|f-g\|_{M_{\beta}}, \label{eq-5.7-k} \\
        \|f \nabla \Delta^{-1} f-g \nabla \Delta^{-1} g\|_{H^{s}} \leq C\left(\|f\|_{H^{s}}+\|f\|_{M_{\beta}}+\|g\|_{H^{s}}\right)\left(\|f-g\|_{M_{\beta}}+\|f-g\|_{H^{s}}\right) \label{eq-5.6-k}, \\
        \|f \nabla \Delta^{-1} f-g \nabla \Delta^{-1} g\|_{M_{\beta}} \leq C\left(\|f\|_{H^{s}}+\|f\|_{M_{\beta}}+\|g\|_{H^{s}}+\|g\|_{M_{\beta}}\right)\left(\|f-g\|_{M_{\beta}}+\|f-g\|_{H^{s}}\right). \label{eq-5.8-k}
    \end{gather}    
     \indent Consequently,  \vspace{-0.3cm}
     \begin{gather}
        \|f^{q}-g^{q}\|_{K_{s, \beta}} \leq C\left(\|f\|_{K_{s, \beta}}^{q-1}+\|g\|_{K_{s, \beta}}^{q-1}\right)\|f-g\|_{K_{s, \beta}}, \label{estimativa-4} \\
        \|f \nabla \Delta^{-1} f-g \nabla \Delta^{-1} g\|_{K_{s, \beta}} \leq C\left(\|f\|_{K_{s, \beta}}+\|g\|_{K_{s, \beta}}\right) \|f-g\|_{K_{s, \beta}}. \label{estimativa-5}
    \end{gather}
\end{lemma} 
\begin{proof}
   Estimate \eqref{eq-5.5-k} follows from writing $f^{q}-g^{q}=(f-g)\left(f^{q-1}+\cdots+g^{q-1}\right)$ and the fact that $H^{s}$ is an algebra when $s>d / 2$. 
   For \eqref{eq-5.7-k}, a similar approach is adopted, using the Sobolev embedding theorem for $s > d/2+ 1$: $\|\varphi\|_{L^{\infty}}+\|\nabla \varphi\|_{L^{\infty}} \leq C\|\varphi\|_{H^{s}}$, for any $\varphi \in H^{s}$. This leads to
    \begin{equation*}
        \|f^{q}-g^{q}\|_{M_\beta} \leq \left( \|f^{q-1}+\cdots+g^{q-1}\|_{L^{\infty}}+\|\nabla \left(f^{q-1}+\cdots+g^{q-1}\right)\|_{L^{\infty}} \right)\|f-g\|_{M_\beta}.
    \end{equation*}

    To establish \eqref{eq-5.6-k}, we start by writing
    \begin{equation}
        \label{eq-5.6-k-2}
        f \nabla \Delta^{-1} f-g \nabla \Delta^{-1} g=(f-g) \nabla \Delta^{-1} f+g\left(\nabla \Delta^{-1} f-\nabla \Delta^{-1} g\right).
    \end{equation}
    Now,  note that, for any $\varphi \in H^{s}$, by decomposing $\nabla \Delta^{-1} \varphi$ into two parts, we find that
    \begin{equation*}
        \|\nabla \Delta^{-1} \varphi\|_{L^{\infty}} \leq C \displaystyle \underset{\quad x \in \mathbb{R}^d}{\operatorname{ess \sup}} \left(\int_{B_1(x)} \frac{|\varphi(y)|}{|x-y|^{d-1}}\mathrm{~d} y+\int_{\mathbb{R}^d  \setminus  B_1(x)} \frac{|\varphi(y)|}{|x-y|^{d-1}}\mathrm{~d} y\right),
    \end{equation*}
    leading to the estimate 
    \begin{equation}
        \label{eq-Lemma-5.2-k}
        \|\nabla \Delta^{-1} \varphi\|_{L^{\infty}} \leq C\left(\|\varphi\|_{L^{\infty}}+\|\varphi\|_{L^1}\right),
    \end{equation}
    which further implies
    \begin{equation}
        \label{eq-Lemma-5.2-k-2}
        \|\nabla \Delta^{-1} \varphi \|_{L^{\infty}} \leq C \left( \|\varphi\|_{H^s} +\|\varphi\|_{M_{\beta}}\right) .
    \end{equation}

    Next, since $s$ is an integer,  using  $\|\varphi\|_{H^s}=\left(\sum_{|\zeta| \leq s}\|\partial^\zeta \varphi\|_2^2\right)^{1 / 2}$, we   have for $\varphi$, $\psi \in H^{s}$
   \begin{equation*}
        \|\psi \nabla \Delta^{-1} \varphi \|_{H^s}  \leq  \sum_{|\zeta| \leq s}   \|\partial^\zeta \psi \nabla \Delta^{-1} \varphi\|_{L^2} 
        +\sum_{\substack{|\zeta+\gamma| \leq s \\ |\gamma|=1}}\|\partial^\zeta \left(\psi \partial^{\gamma}\left(\nabla \Delta^{-1} \varphi\right) \right)\|_{L^2}=I_1+I_2.
    \end{equation*} 
    Using \eqref{eq-Lemma-5.2-k-2}, we estimate $I_1$:
    \begin{equation*}
        I_1 \leq \sum_{|\zeta| \leq s}  \|\partial^\zeta \psi\|_{L^2} \| \nabla \Delta^{-1} \varphi \|_{L^\infty}
        \leq  \sum_{|\zeta| \leq  s} C \|\partial^\zeta \psi\|_{L^2} \left( \|\varphi\|_{H^s} +\|\varphi\|_{M_{\beta}}\right).
    \end{equation*}
    By the Sobolev embedding theorem and \citet[Lemma A.1]{INEQUALITY}, we have $I_2 \leq C\|\psi \|_{H^s} \| \varphi \|_{H^s}$. Thus, we obtain 
   \begin{equation}
        \label{eq-Lemma-5.2-k-3}
        \|\psi \nabla \Delta^{-1} \varphi \|_{H^s} \leq   C\left( \|\varphi\|_{H^s} +\|\varphi\|_{M_{\beta}}\right)\|\psi\|_{H^s}.
    \end{equation}
    Then, from \eqref{eq-Lemma-5.2-k-3} and \eqref{eq-5.6-k-2}, we derive %
    \begin{equation*}
        \begin{split}
            \|f \nabla \Delta^{-1} f-g \nabla \Delta^{-1} g\|_{H^{s}} 
            & \leq C \left( \|f\|_{H^{s}} +\|f\|_{M_{\beta}}\right) \|f-g\|_{H^{s}}+C\left( \|f-g\|_{H^{s}} +\|f-g\|_{M_{\beta}}\right) \|g\|_{H^{s}}\\
            & \leq C\left(\|f\|_{H^{s}}+\|f\|_{M_{\beta}}+\|g\|_{H^{s}}\right)\left(\|f-g\|_{M_{\beta}}+\|f-g\|_{H^{s}}\right). 
        \end{split}
    \end{equation*}

    For \eqref{eq-5.8-k}, we observe that
    \begin{equation}
        \label{eq-5.6-k-3}
        \begin{split}
            \nabla \cdot\left(f \nabla \Delta^{-1} f-g \nabla \Delta^{-1} g\right) & =\left(\nabla f \cdot \nabla \Delta^{-1} f-\nabla g \cdot \nabla \Delta^{-1} g\right)+\left(f^{2}-g^{2}\right)\\
            &=\nabla (f-g) \cdot \nabla \Delta^{-1} f+\nabla g \cdot \nabla \Delta^{-1} \left(f- g\right)+\left(f-g\right)\left(f+g\right).
        \end{split}
    \end{equation}

     Then, by summing \eqref{eq-5.6-k-2} and \eqref{eq-5.6-k-3}, integrating it against $\left(1+|x|^{\beta}\right)$, and applying estimate \eqref{eq-Lemma-5.2-k-2} along with the Sobolev embedding theorem, we see that \eqref{eq-5.8-k} does not exceed 
     \begin{equation*}
         \begin{split}
             & \|f-g\|_{M_{\beta}}\|\nabla \Delta^{-1} f\|_{L^{\infty}}+\|g\|_{M_{\beta}}\|\nabla \Delta^{-1}(f-g)\|_{L^{\infty}} + \|f-g\|_{M_{\beta}} \|f+g\|_{L^{\infty}}  \\
             & \quad  \leq C\left(\|f\|_{H^{s}}+\|f\|_{M_{\beta}}+\|g\|_{H^{s}}+\|g\|_{M_{\beta}}\right)\left(\|f-g\|_{M_{\beta}}+\|f-g\|_{H^{s}}\right). 
         \end{split}
     \end{equation*}

   Finally, note that \eqref{estimativa-4} and \eqref{estimativa-5} naturally follow from the definition of the norm $\|\, \cdot \,\|_{K_{s, \beta}}$ and the previous computations.
\end{proof}

\section{Global Existence of smooth solutions} \label{sec:Global-Existence}

To establish the global existence of smooth solutions, our initial step involves the construction of a local solution in the Banach space $X_{s, \beta}^T \equiv C\left(K_{s, \beta},[0, T]\right)$, where $T>0$ is chosen sufficiently small. To this purpose, from \eqref{eq-5.4-k}, let us denote
\begin{equation*}
    B_t(\rho) \equiv \int_0^t K^{\alpha}_{t-r} *\left(-\nabla \cdot(u \rho)-\epsilon \rho^q+\chi \nabla \cdot\left(\rho \nabla \Delta^{-1} \rho\right)\right) \mathrm{d} r.
\end{equation*}
\begin{lemma} \label{Lemma-5.3-k}
    Assume $u \in C^{\infty}(\mathbb{R}^d \times[0, \infty))$ satisfies $\nabla \cdot u=0$. Let $s$ and $q$ be positive integers, where $s>d/2+1$, and let $\alpha$ and $\beta$ be such that $\alpha \in (1,2]$, and $0 \leq \beta < \alpha$. Then, for any $f, g \in X_{s, \beta}^T$, we have 
    \begin{equation*}
        \|B_T(f)-B_T(g)\|_{X_{s, \beta}^T} \leq  \Theta \|f-g\|_{X_{s, \beta}^T},
    \end{equation*}
    where, for $T \leq 1$, we have
    \begin{equation*}
        \Theta  \leq C(d, q, \beta , \epsilon, \chi) \max _{0 \leq t \leq T}  \left(\|u(\cdot, t)\|_{C^s}+\|f(\cdot, t)\|_{K_{s, \beta}}^{q-1}+  \|g(\cdot, t)\|_{K_{s, \beta}}^{q-1} + \|f(\cdot, t)\|_{K_{s, \beta}}+\|g(\cdot, t)\|_{K_{s, \beta}}\right) T^{\frac{\alpha-1}{\alpha}}. 
    \end{equation*}
\end{lemma}
\begin{proof}
    Consider
    \begin{equation*}
        B_t(f)-B_t(g)=\int_0^t K^{\alpha}_{t-r} *\left(\nabla(u(f-g))-\epsilon\left(f^q-g^q\right)+\chi \nabla \cdot \left(f \nabla \Delta^{-1} f-g \nabla \Delta^{-1} g\right)\right) \mathrm{d} r .
    \end{equation*}

    From \Cref{Lemma-5.1-k,Lemma-5.2-k},  and the fact that  $\alpha > 1$, we obtain 
    \begin{equation*}
        \begin{split}
            \|B_t(f)-B_t(g)\|_{K_{s, \beta}} & \leq C \int_0^t\left[\left((t-r)^{-1 / \alpha}+(t-r)^{(\beta-1) / \alpha}\right)\left(\|u\|_{C^s}+\|f\|_{K_{s, \beta}}+\|g\|_{K_{s, \beta}}\right)\right. \\
            & \hspace{1.8cm} \left.+\left(1+(t-r)^{\beta / \alpha}\right)\left(\|f\|_{K_{s, \beta}}^{q-1}+\|g\|_{K_{s, \beta}}^{q-1}\right)\right]\|f-g\|_{K_{s, \beta}} \mathrm{d} r \\
            & \leq C\left[\left(t^{\frac{\alpha-1}{\alpha}}+t^{\frac{\alpha-1+\beta}{\alpha}}\right) \max _{0 \leq r \leq t}\left(\|u\|_{C^s}+\|f\|_{K_{s, \beta}}+\|g\|_{K_{s, \beta}}\right)\right. \\
            & \hspace{1.8cm} \left.+\left(t+t^{\frac{\beta+\alpha}{\alpha}}\right) \max _{0 \leq r \leq t}\left(\|f\|_{K_{s, \beta}}^{q-1}+\|g\|_{K_{s, \beta}}^{q-1}\right)\right] \max _{0 \leq r \leq t}\|f-g\|_{K_{s, \beta}},
        \end{split}        
    \end{equation*}
    and considering $T \leq 1$ we can neglect the higher powers of $T$, yielding the estimate for $\Theta$.
\end{proof}

In a standard way, the existence of a local solution is implied by \Cref{Lemma-5.3-k} through the contraction mapping principle. This leads us to formulate the following theorem. 

\begin{theorem} \label{Theorem-5.4-k}
    Consider positive integers $q$ and $s$ such that $s>d/2+1$. Let $\alpha \in (1,2]$, $0 \leq  \beta  < \alpha$, and assume $u \in C^{\infty}\left(\mathbb{R}^{d} \times [0, \infty)\right)$ with $\nabla \cdot u=0$. Additionally, suppose $\rho_{0} \in K_{s, \beta}$. Then, there exists $T=T\left(q, d, u, s, \epsilon, \chi, \alpha, \|\rho_{0}\|_{K_{s, \beta}}\right)$ such that equation \eqref{eq-4.1-k} has a unique mild solution $\rho \in X_{s,  \beta }^{T}$ satisfying $\rho(x, 0)=\rho_{0}(x)$.
\end{theorem}
\begin{remark}
    Higher regularity of the solution in space and time -- implying $\rho \in C\left(H^{m},(0, T]\right)$ for every $m>0$ -- follows from \Cref{Theorem-5.4-k} and parabolic regularity estimates applied iteratively \citep{New-developments-nonlocal-operators-I,Regularity-solutions-the-parabolic-fractional-obstacle}. 
\end{remark}
\begin{remark}
    Note that, from \Cref{Lemma-5.3-k},  the value of $\Theta$ depends on the $H^{s}$ and $M_{\beta}$ norms of functions in $X_{\beta, s}^{T}$. Therefore, under the conditions of  \Cref{Theorem-5.4-k}, if there is a control over the growth of the $H^{s}$ and $M_{\beta}$ norms of the solution, the time existence of the solution can be extended through iterative applications of local results. Hence, to prove global existence of smooth solutions, we can establish a global a priori estimate on $\|\rho(\cdot, t)\|_{H^{s}}$ and $\|\rho(\cdot, t)\|_{M_{\beta}}$, and then the local solution can be extended globally to $X_{\beta, s}^{T}$ with arbitrary $T$.
\end{remark}

To prove the bounds on $M_{\beta}$ and $H^{s}$ norms of the solution, we first establish control of the $L^{\infty}$ norm. 

\begin{lemma} \label{Lemma-5.6-k}
    Assume that $\rho$ is the local mild solution guaranteed by \Cref{Theorem-5.4-k}, and  $q>2$. Then, 
    \begin{equation}
        \label{eq-5.12-k}
        \|\rho(\cdot, t)\|_{L^{\infty}} \leq N_0 \equiv \max \left((\chi / \epsilon)^{\frac{1}{q-2}},\|\rho_0\|_{L^{\infty}}\right)
    \end{equation}
    for all $0 \leq t \leq T$.
\end{lemma}
\begin{proof}
For the sake of completeness, we present the same proof found in \cite[Lemma 5.6]{Kiselev-Biomixing}, this time addressing the fractional Laplacian.
    
Suppose that the statement is false and that there exist $N_1>N_0$ and $0<t_1 \leq T$ such that $\left\|\rho\left(x,t_1\right)\right\|_{L^{\infty}}=N_1$ is attained for the first  time, \ie $\left|\rho\left(x, t\right)\right| \leq\left.N_1\right.$ holds for every $x$ and $0 \leq t \leq t_1$.

In this case, there exists $x_0$ such that $\rho\left(x_0, t_1\right)=N_1$. Indeed, let $x_k$ be a sequence such that $\rho\left(x_k, t_1\right) \rightarrow N_1$ as $k \rightarrow \infty$. If $x_k$ has finite accumulation points, one of them can be labeled $x_0$ and, by continuity, $\rho\left(x_0, t_1\right)=N_1$.
Otherwise, if $x_k \rightarrow \infty$, we can consider a subsequence and assume the unit balls around $x_k$, $B_1\left(x_k\right)$, are disjoint; then, using a version of Poincare inequality (e.g., \cite{Poincare-inequality}), we see that $\|\rho-\bar{\rho}\|_{L^{\infty}\left(B_1\left(x_k\right)\right)}^2 \leq C\|\rho\|_{H^s\left(B_1\left(x_k\right)\right)}^2$ and, since $\sum_k\|\rho\|_{H^s\left(B_1\left(x_k\right)\right)}^2 \leq C\left(t_1\right)<\infty$,
$$
\bar{\rho}_k \equiv \frac{1}{\left|B_1\left(x_k\right)\right|} \int_{B_1\left(x_k\right)} \rho d x \xrightarrow{k \rightarrow \infty} N_1,
$$
which is impossible as $\int_{\mathbb{R}^d}|\rho(x)|\left(1+|x|^{\beta}\right) d x \leq C\left(t_1\right)$.

Then, considering $x_0$ such that $\rho\left(x_0, t_1\right)=N_1$ is a maximum (the case of a minimum follows an analogous procedure), we obtain 
\begin{equation*}
    \Lambda^\alpha \rho(x_0, t_1)  = c_{d,\alpha} P . V . \int _{\mathbb{R}^d} \frac{\rho(x_0, t_1)-\rho(y, t_1) }{|x_0-y|^{\alpha+d}} \mathrm{~d} y \geq 0 \; (\leq 0) \quad \text{ for all } y \in \mathbb{R}^d,
\end{equation*}
implying
\begin{equation*}
    \begin{aligned}
    \left.\partial_t \rho\left(x_0, t\right)\right|_{t=t_1}= & (u \cdot \nabla) \rho\left(x_0, t_1\right)-\Lambda^\alpha \rho\left(x_0, t_1\right)+\chi \nabla \rho\left(x_0, t_1\right) \cdot \nabla \Delta^{-1} \rho\left(x_0, t_1\right) \\
    & +\chi \rho\left(x_0, t_1\right)^2-\epsilon \rho\left(x_0, t_1\right)^q \leq \rho\left(x_0, t_1\right)^2\left(\chi-\epsilon \rho\left(x_0, t_1\right)^{q-2}\right).
    \end{aligned}
\end{equation*}
Therefore, by the assumption on $N_1$, we have $\partial_t \rho\left(x_0, t_1\right)<0$, which contradicts our choice of $t_1$.
\end{proof}

Next, we prove an upper bound on the growth of the $M_{\beta}$ norm of the solution.

\begin{lemma} \label{lemma-5.7-k}
    Assume that $\rho$ is the local mild solution guaranteed by \Cref{Theorem-5.4-k}. Then, on the interval of existence, we have the following bound for the growth of the $M_{\beta}$ norm of the solution 
    \begin{multline}
        \label{eq-5.13-k}
        \|\rho(\cdot, t)\|_{M_{\beta}} \leq C\left(1+t^{\frac{\beta}{\alpha}}\right)\|\rho_0\|_{M_{\beta}} 
        \exp \left(C \int_{0}^{t} \left[\epsilon\left(1+\left(t-r\right)^{\frac{\beta}{\alpha}}\right)\|\rho(\cdot, r)\|_{L^{\infty}}^{q-1} \right. \right.\\
        \left. \textcolor{white}{\int_{0}^{t}}  \left.+\left(\left(t-r\right)^{-\frac{1}{\alpha}}+\left(t-r\right)^{\frac{\beta-1}{\alpha}}\right)\left(\|u(\cdot, r)\|_{C^1}+\|\rho(\cdot, r)\|_{L^{\infty}}+\|\rho(\cdot, r)\|_{L^1}\right)\right] \mathrm{d} r\right).
    \end{multline}
\end{lemma}
\begin{proof}
    Note that the following estimates hold:
    \begin{equation}
        \label{estimates-lemma-5.7-k}
        \begin{aligned}
            \|u \rho\|_{M_{\beta}} & \leq\|u\|_{C^1}\|\rho\|_{M_{\beta}}, \\
            \|\rho^q\|_{M_{\beta}} & \leq C \|\rho\|_{L^{\infty}}^{q-1}\|\rho\|_{M_{\beta}}, \\
            \|\rho \nabla \Delta^{-1} \rho\|_{M_{\beta}} & \leq C\left(\|\rho\|_{L^{\infty}}+\|\rho\|_{L^1}\right)\|\rho\|_{M_{\beta}},
        \end{aligned}
    \end{equation}
    where the last inequality is due to
    \begin{equation*}
        \begin{split}
            \|\rho \nabla \Delta^{-1} \rho\|_{M_{\beta}} & =\int_{\mathbb{R}^d} \left[|\nabla \Delta^{-1} \rho| (|\rho(x)|+|\nabla \rho(x)|) +|\rho(x)|^2\right] \left(1+|x|^{\beta}\right) \mathrm{~d} x  
            \\ &
            \leq \left(\|\nabla \Delta^{-1} \rho\|_{L^{\infty}} + \| \rho\|_{L^{\infty}}\right)\|\rho\|_{M_{\beta}}
        \end{split}
    \end{equation*}
    and estimate \eqref{eq-Lemma-5.2-k}. 

    Now, employing estimates \eqref{estimativa-1-k} 
    and \eqref{estimates-lemma-5.7-k} to \eqref{eq-5.4-k}, we obtain 
    \begin{multline}
        \label{eq-5.14-k}
        \|\rho(\cdot, t)\|_{M_{\beta}} \leq C\left(1+t^{\frac{\beta}{\alpha}}\right)\|\rho_0\|_{M_{\beta}} 
        +C \int_{0}^{t} \left[\epsilon\left(1+\left(t-r\right)^{\frac{\beta}{\alpha}}\right)\|\rho(\cdot, r)\|_{L^{\infty}}^{q-1} \right. \\
         \left.+\left(\left(t-r\right)^{-\frac{1}{\alpha}}+\left(t-r\right)^{\frac{\beta-1}{\alpha}}\right)\left(\|u(\cdot, r)\|_{C^1}+\|\rho(\cdot, r)\|_{L^{\infty}}+\|\rho(\cdot, r)\|_{L^1}\right)\right] \|\rho(\cdot, r)\|_{M_{\beta}}  \mathrm{d} r.
    \end{multline}

    Then, applying Gronwall's inequality to \eqref{eq-5.14-k}, noticing that $(1+t^{\beta/\alpha})$ is a non-decreasing function on $t$, we obtain \eqref{eq-5.13-k}.
\end{proof}

Now we prove uniform in time bounds on the $H^s$ norm of the solution.

\begin{lemma}
    Let $\rho$ be the local mild solution whose existence is guaranteed by \Cref{Theorem-5.4-k}  and suppose $\|\rho(\cdot, t)\|_{L^{\infty}}$ does not exceed $N_0$ for all $0 \leq t \leq T$. Then, for $s$ even, we have
    \begin{equation*}
        \|\rho(\cdot, t)\|_{H^s} \leq \max \left(\|\rho_0\|_{H^s}, C\left(u, d, q, s, \chi, \epsilon, N_0\right)\right).
    \end{equation*}
\end{lemma} 
\begin{proof}
    Here we follow the same steps as in \cite[Lemma 5.8]{Kiselev-Biomixing}.  Applying $\Delta^{s / 2}$ to \eqref{eq-4.1-k},  multiplying by $\Delta^{s / 2} \rho(x, t)$, and integrating, we find
    \begin{multline}
        \label{eq-5.16-k}
        \frac{1}{2} \frac{\mathrm{d}}{\mathrm{d} t}\|\rho\|_{\dot{H}^s}^2=\int_{\mathbb{R}^d} \Delta^{s / 2} \left[(u \cdot \nabla) \rho\right]\left(\Delta^{s / 2} \rho\right) \mathrm{~d} x-\int_{\mathbb{R}^d}\left(\Delta^{(s+\alpha) / 2} \rho\right)\left(\Delta^{s / 2} \rho\right) \mathrm{~d} x \\
        -\epsilon \int_{\mathbb{R}^d}\left(\Delta^{s / 2} \rho^q\right)\left(\Delta^{s / 2} \rho\right) \mathrm{~d} x
        +\chi \int_{\mathbb{R}^{d}}\left[\nabla \cdot \Delta^{s / 2}\left(\rho \nabla \Delta^{-1} \rho\right)\right]\left(\Delta^{s / 2} \rho\right) \mathrm{~d} x.
    \end{multline}
    
   Now, using the fact that $\nabla \cdot u=0$, for the first integral on the right-hand side of \eqref{eq-5.16-k}, we obtain
   \begin{equation*}
       \left|\int_{\mathbb{R}^{d}} \Delta^{s / 2} \left[(u \cdot \nabla) \rho\right]\left(\Delta^{s / 2} \rho\right) \mathrm{~d} x\right| \leq C\|u\|_{C^{s}}\|\rho\|_{\dot{H}^{s}}^{2}.
   \end{equation*}
 
  For the second one, we see that 
    \begin{equation*}
        \int_{\mathbb{R}^d}\left(\Delta^{(s+\alpha) / 2} \rho\right)\left(\Delta^{s / 2} \rho\right) \mathrm{~d} x=\|\rho\|_{\dot{H}^{s+\frac{\alpha}{2}}}^2.
    \end{equation*}

   The third integral can be written as a sum of a finite number of terms of the form $ \int_{\mathbb{R}^{d}} D^{s} \rho \; \prod_{i=1}^{q} D^{s_{i}} \rho \mathrm{d} x$, $s_{1}+\cdots+s_{q}=s$, $s_{i} \geq 0$, where $D^{l}$ denotes any partial derivative operator of the $l\text{th}$ order. By Hölder's inequality, we have
    \begin{equation*}
        \left|\int_{\mathbb{R}^{d}} D^{s} \rho \prod_{i=1}^{q} D^{s_{i}} \rho \mathrm{~d} x\right| \leq\|D^{s} \rho\|_{L^{2}} \prod_{i=1}^{q}\|D^{s_{i}} \rho\|_{p_{i}},
    \end{equation*}
    $\displaystyle \sum_{i=1}^{q} 1/p_{i}=1 / 2$ . Then, taking $p_{i}=2 s / s_{i}$ and using Gagliardo-Nirenberg inequality, we obtain 
    \begin{equation}
        \label{eq-5.17-k}
        \|D^{s_{i}} \rho\|_{L^{2 s / s_{i}}} \leq C\|\rho\|_{L^{\infty}}^{1-\frac{s_{i}}{s}}\|D^{s} \rho\|_{L^{2}}^{\frac{s_{i}}{s}},
    \end{equation}
    and hence
    \begin{equation*}
        \left|\int_{\mathbb{R}^{d}} \Delta^{s / 2} \rho^{q} \Delta^{s / 2} \rho \mathrm{~d} x\right| \leq C\|\rho\|_{L^{\infty}}^{q-1}\|\rho\|_{\dot{H}^{s}}^{2}.
    \end{equation*}

    Additionally, to express the fourth integral, we can use a sum of a finite number of terms of the form $ \int_{\mathbb{R}^{d}} D^{s} \rho D^{k} \rho D^{s+2-k} \Delta^{-1} \rho \mathrm{~d} x$, where $k=0, \ldots, s$.  The only term one gets from the direct differentiation that does not appear to be of this form is $ \int_{\mathbb{R}^{d}} \Delta^{s / 2} \rho \nabla \Delta^{s / 2} \rho \nabla \Delta^{-1} \rho \mathrm{~d} x$. However, we find that this term is equal to $ -\frac{1}{2} \int_{\mathbb{R}^{d}} \rho \left|\Delta^{s / 2} \rho\right|^{2}  \mathrm{~d} x$ through integration by parts.
    %
    %
    Now,
    \begin{equation*}
        \left|\int_{\mathbb{R}^{d}} D^{s} \rho D^{k} \rho D^{s+2-k} \Delta^{-1} \rho \mathrm{~d} x\right| \leq C\|D^{s} \rho\|_{L^{2}}\|D^{k} \rho\|_{L^{p_{1}}}\|D^{s-k} \rho\|_{L^{p_{2}}},
    \end{equation*}
    $1/p_{1}+1/p_{2}=1 / 2$, $p_{2}<\infty$, where 
    we used boundedness of Riesz transforms on $L^{p_{2}}$, $p_{2}<\infty$. Then, setting $p_{1}=\frac{2 s}{k}$, $p_{2}=\frac{2 s}{s-k}$, and using Gagliardo-Nirenberg inequality \eqref{eq-5.17-k} with $s_{i}=$ $k, s-k$, we get
    \begin{equation*}
        \left|\int_{\mathbb{R}^{d}} D^{s} \rho D^{k} \rho D^{s+2-k} \Delta^{-1} \rho \mathrm{~d} x\right| \leq C\|\rho\|_{L^{\infty}}\|\rho\|_{\dot{H}^{s}}^{2}.
    \end{equation*}
    
   Thus, putting all the estimates into \eqref{eq-5.16-k}, we find that
    \begin{equation}
        \label{eq-5.18-k-1}
        \frac{1}{2} \frac{\mathrm{d}}{\mathrm{d} t}\|\rho\|_{\dot{H}^s}^2 \leq C\|\rho\|_{L^{\infty}}\|\rho\|_{\dot{H}^s}^2-\|\rho\|_{\dot{H}^{s+\frac{\alpha}{2}}}^2,       
    \end{equation}
   and from Hölder's inequality 
   we obtain $\|\rho\|_{\dot{H}^s}\leq C\|\rho\|_{\dot{H}^{s+\frac{\alpha}{2}}}^{\frac{s}{s+\alpha/2}}\|\rho\|_{L^{2}}^{\frac{\alpha/2}{s+\alpha/2}}$. 
    Then, we can rewrite \eqref{eq-5.18-k-1} as
    \begin{equation}
        \label{eq-5.18-k-2}
        \frac{1}{2} \frac{\mathrm{d}}{\mathrm{d} t}\|\rho\|_{\dot{H}^s}^2 \leq C\|\rho\|_{L^{\infty}}\|\rho\|_{\dot{H}^{s+\frac{\alpha}{2}}}^{\frac{2s}{s+\alpha/2}}\|\rho\|_{L^{2}}^{\frac{\alpha}{s+\alpha/2}}-\|\rho\|_{\dot{H}^{s+\frac{\alpha}{2}}}^2.       
    \end{equation} 

    Note that, as $\|\rho\|_{L^{2}}\leq \|\rho\|_{L^{1}}^{1/2} \|\rho\|_{L^{\infty}}^{1/2}< \infty$, the differential inequality \eqref{eq-5.18-k-2} implies the result of the lemma, since if $C\|\rho\|_{L^{\infty}}\|\rho\|_{\dot{H}^{s+\frac{\alpha}{2}}}^{\frac{2s}{s+\alpha/2}}\|\rho\|_{L^{2}}^{\frac{\alpha}{s+\alpha/2}}-\|\rho\|_{\dot{H}^{s+\frac{\alpha}{2}}}^2>0$, then
    \begin{equation*}
         \|\rho\|_{\dot{H}^s}\leq \|\rho\|_{\dot{H}^{s+\frac{\alpha}{2}}} < \|\rho\|_{L^{\infty}}^{\frac{s-\alpha / 2}{s}} \|\rho\|_{L^{2}} \leq C \|\rho\|_{L^{1}} N_0^{\frac{3}{2}-\frac{\alpha}{2s}},
    \end{equation*}
    whereas if $C\|\rho\|_{L^{\infty}}\|\rho\|_{\dot{H}^{s+\frac{\alpha}{2}}}^{\frac{2s}{s+\alpha/2}}\|\rho\|_{L^{2}}^{\frac{\alpha}{s+\alpha/2}}-\|\rho\|_{\dot{H}^{s+\frac{\alpha}{2}}}^2\leq 0$, then $\|\rho\|_{\dot{H}^s}\leq \|\rho_0\|_{\dot{H}^s}$.
\end{proof}

Given the $M_{\beta}$ and $H^s$ norms bounds, we proved that the local solution can now be continued globally, establishing the following theorem: 
\begin{theorem}[\textbf{Global existence of smooth solutions}] \label{Theorem-4.1-k}
    Let $q>2$, $s>d / 2+1$ be integers, $\alpha \in (1,2]$, and $\beta$ satisfy $0 \leq \beta < \alpha$. Assume $u \in C^{\infty}\left(\mathbb{R}^{d} \times[0, \infty)\right)$ is divergence free, and $\rho_{0} \in K_{s, \beta}$. Then, there exists a unique mild solution $\rho$ to equation \eqref{eq-4.1-k} in $C\left(K_{s, \beta},[0, \infty)\right) \cap C^{\infty}\left(\mathbb{R}^{d} \times(0, \infty)\right)$.
\end{theorem}
%

Observe that, for $\rho_0 \geq 0$, $\rho$ in equation \eqref{eq-4.1-k} describes the densities of sperm and egg gametes, and the biological interest lies in nonnegative solutions, $\rho \geq 0$. Then, before concluding this section, we present a result that ensures the solution guaranteed by \Cref{Theorem-5.4-k} or \Cref{Theorem-4.1-k}, corresponding to a nonnegative initial condition, remains nonnegative.

\begin{theorem} \label{Theorem-nonnegative-solution}
    Let $q>2$, $s>d / 2+1$ be integers, $\alpha \in (1,2]$, and $\beta$ satisfy $0 \leq \beta < \alpha$. Suppose that $u \in C^{\infty}\left(\mathbb{R}^{d} \times[0, \infty)\right)$ is divergence free and $\rho_{0} \in K_{s, \beta}$ is nonnegative. Then, the solution $\rho$ guaranteed by \Cref{Theorem-5.4-k} or \Cref{Theorem-4.1-k} remains nonnegative for all $x$ and $t$.
\end{theorem}
\begin{proof}
    The proof relies on the inequality $ \displaystyle \int_{\mathbb{R}^d}\rho_{-} \Lambda^\alpha \rho \mathrm{~d} x \geq 0$, where $\rho_{-}(x,t)=\min (\rho(x,t), 0)$, alongside standard arguments. 
\end{proof}
 
\section{Reaction efficiency} \label{sec:Reaction-efficiency}

In this section, we aim to analyze the dynamics of the total fraction of unfertilized eggs,  denoted as $m(t)$ and given by \eqref{unfertilized-eggs}, in both chemotactic and chemotaxis-free scenarios.
Our investigation delves into understanding the impact of chemotaxis on reaction efficiency by comparing these specific cases.

It is important to emphasize that, for a nonnegative initial condition $\rho_0\geq 0$, the quantity $m(t)$ corresponds to the $L^{1}$ norm of $\rho$. This is a consequence of \Cref{Theorem-nonnegative-solution}, which establishes the nonnegativity of the solution provided by either \Cref{Theorem-5.4-k} or \Cref{Theorem-4.1-k} in such case.
We also highlight that $m(t)$ exhibits a monotone decreasing behavior. To reinforce this statement, we next enhance the proof presented in  \Cref{monotone-decreasing-fraction-unfertilized-eggs} and provide a rigorous justification of this observation. 
\begin{Proposition} \label{Proposition-monotone-decreasing-behavior}
    Consider the nonnegative initial condition $\rho_0 \geq 0$, and let $\rho$ be the mild solution provided by \Cref{Theorem-5.4-k} or \Cref{Theorem-4.1-k}. Then, it follows that the total fraction of unfertilized eggs by time $t$, given by \eqref{unfertilized-eggs}, exhibits a monotone decreasing behavior.
\end{Proposition}
\begin{proof}
    Integrating both sides of equation \eqref{eq-5.4-k} and applying Tonelli's theorem, Fubini's theorem, and \Cref{integral-k}\ref{estimate-k(c)}, we obtain
    \begin{equation*}
        \int_{\mathbb{R}^d} K^{\alpha}_{t}* \rho_{0}(x) \mathrm{~d} x= \int_{\mathbb{R}^d}\rho_{0}(y) \left(\int_{\mathbb{R}^d} K^{\alpha}_{t}(x-y)   \mathrm{~d} x\right) \mathrm{~d} y= \int_{\mathbb{R}^d}\rho_{0}(y) \mathrm{~d} y, 
    \end{equation*}
    as $\rho_{0} \in L^1$. 
    Similarly, we have
    \begin{equation*}
       \int_0^t  \int_{\mathbb{R}^d} K^{\alpha}_{t-s}* \rho^q \mathrm{~d} x \mathrm{d} s= \int_0^t  \int_{\mathbb{R}^d}\rho^q\left(y,s\right)  \left(\int_{\mathbb{R}^d} K^{\alpha}_{t-s}(x-y)   \mathrm{~d} x\right) \mathrm{~d} y \mathrm{d} s=\int_0^t \int_{\mathbb{R}^d}\rho^q \left(y,s\right) \mathrm{~d} y \mathrm{d} s , 
    \end{equation*}
    as $\rho^q(x,t) \in L^1$: $ \| \rho^q(\cdot,t) \|_{L^1} \leq \| \rho(\cdot,t)\|_{L^{\infty}}^{q-1} \| \rho(\cdot,t)\|_{L^1}$. 
    Additionally, for $v=-u \rho+\chi \rho \nabla \Delta^{-1} \rho \in L^1$, we have
    \begin{equation*}
        \int_{\mathbb{R}^d} \nabla K^{\alpha}_{t-s} * v(x,s)  \mathrm{~d} x = \int_{\mathbb{R}^d} v(y,s) \left(\int_{\mathbb{R}^d} \nabla K^{\alpha}_{t-s}(x-y)   \mathrm{~d} x\right) \mathrm{~d} y=0. 
    \end{equation*}
    Note, as proved before, $\| v\|_{L^1} \leq \|u\|_{C^1} \|\rho\|_{L^1} +\|\nabla \Delta^{-1} \rho\|_{L^{\infty}} \|\rho\|_{L^1} \leq C \left(\|u\|_{C^1}+ \|\rho\|_{L^{\infty}}+\|\rho\|_{L^1}\right)\|\rho\|_{L^1}$.  

    Therefore, we establish that 
    \begin{equation*}
        \int_{\mathbb{R}^d} \rho (x, t) \mathrm{~d} x=\int_{\mathbb{R}^d} \rho_0(x) \mathrm{~d} x- \epsilon \int_0^t \int_{\mathbb{R}^d}\rho^q (x,s) \mathrm{~d} x \mathrm{d} s 
    \end{equation*}
    for every $t \geq 0$ where  the solution exists.  Equivalently, we can write 
    \begin{equation}
        \label{monotone-decreasing}
        \frac{\mathrm{d}}{\mathrm{d} t}\int_{\mathbb{R}^d} \rho (x, t) \mathrm{~d} x=- \epsilon \int_{\mathbb{R}^d}\rho^q (x,t) \mathrm{~d} x.
    \end{equation}   
\end{proof}

The following lemma provides an inequality that 
will be important in establishing a lower bound for the $L^1$ norm of $\rho$ in both chemotactic and chemotaxis-free scenarios.
\begin{lemma}{\citep{A-maximum-principle-Cordoba,Guo-2015-Partial-Differential-Equations}} \label{Cordoba_Cordoba}
    Let $\alpha \in[0,2]$ and $f, \Lambda^\alpha f \in L^p$. Then, for any $p \geq 2$, it is true that
    \begin{equation}
        \label{Positivity-Lemma-111}
        \int_{\mathbb{R}^d}|f|^{p-2} f \Lambda^\alpha f \mathrm{~d} x \geq \frac{2}{p} \int_{\mathbb{R}^d}\left(\Lambda^{\frac{\alpha}{2}}|f|^{\frac{p}{2}}\right)^2 \mathrm{~d} x .
    \end{equation}
\end{lemma}

\subsection{Reaction in a Chemotaxis-Free Environment} 
\label{The-Reaction-advection-diffusion-Case}

Here, we establish that, in the chemotaxis-free scenario, there exists a constant $C_0=C_0(\epsilon,q,\alpha,d,\rho_0)$, independent of $u$, such that $m(t) \geq C_0$ for all $t \geq 0$. Additionally, when $\rho_0$, $u$ and $q$ are fixed, the quantity $m(t)$ converges to $m_0$ as $\epsilon \rightarrow 0$.
To demonstrate this, consider equation \eqref{eq-4.1-k} with $\chi=0$:
\begin{equation}
    \label{eq-2.1-k}
    \left\{\begin{array}{l}
    \partial_{t} \rho+u \cdot \nabla \rho=-\Lambda^\alpha \rho-\epsilon \rho^{q}, \quad \alpha \in (1,2] \\
    \rho(x, 0)=\rho_0(x), \quad x \in \mathbb{R}^d \quad d \geq 2.
    \end{array}\right.
\end{equation}

Note that, by the comparison principle, $\rho \leq b$, where  
\begin{equation}
    \label{eq-2.2-k}
    \left\{\begin{array}{l}
    \partial_{t} b+u \cdot \nabla b=-\Lambda^\alpha b, \quad \alpha \in (1,2] \\
    b(x, 0)=\rho_{0}(x), \quad x \in \mathbb{R}^d \quad d \geq 2.
   \end{array}\right.
\end{equation}

Also, since $\rho \geq 0$, we have
\begin{equation*}
    \frac{\mathrm{d}}{\mathrm{d} t}\|\rho(\cdot, t)\|_{L^{1}}=\frac{\mathrm{d}}{\mathrm{d} t} \int_{\mathbb{R}^{d}} \rho(x, t) \mathrm{~d} x=-\epsilon \int_{\mathbb{R}^{d}} \rho^{q}(x, t) \mathrm{~d} x \geq-\epsilon \int_{\mathbb{R}^{d}} b^{q}(x, t) \mathrm{~d} x.
\end{equation*}

The behavior of the $L^{q}$ norm of $b$ can be used for estimating the decay of the $L^{1}$ norm of $\rho$. 

\begin{lemma} \label{lemma-2.1-k}
There exists $C=C(d,\alpha)$ (independent of the flow $u$) such that
\begin{equation}
    \label{eq-2.3-k}
    \|b(\cdot, t)\|_{L^{2}} \leq \min \left(\|b_{0}\|_{L^{2}}, C t^{-d / 2 \alpha}\|b_{0}\|_{L^{1}}\right), \quad\|b(\cdot, t)\|_{L^{\infty}} \leq \min \left(\|b_{0}\|_{L^{\infty}}, C t^{-d / \alpha}\|b_{0}\|_{L^{1}}\right).
\end{equation} 
\end{lemma}
\begin{proof} 
We proceed by multiplying equation \eqref{eq-2.2-k} by the variable $b$, integrating the resulting equation, and using the incompressibility of the velocity field $u$ to obtain 
\begin{equation}
    \label{eq-Lemma-6.3-1}
    \frac{1}{2} \frac{\mathrm{d}}{\mathrm{d} t}\|b\|_{L^{2}}^{2} =-\| \Lambda^{\frac{\alpha}{2}} b \|_{L^2}^2,
\end{equation}
where we use the Plancherel theorem to write 
\begin{equation*}
    \int_{\mathbb{R}^d} b \Lambda^{\alpha} b \mathrm{~d} x = \| \Lambda^{\frac{\alpha}{2}} b \|_{L^2}^2,
\end{equation*}
since the fractional Laplacian, $\Lambda^{\alpha}$, is defined in Fourier variables as \eqref{eq-definition-fractional-Laplacian-fourier}.  
Now, consider the Riesz potential $I_{\sfrac{\alpha}{2}}=\Lambda^{-\frac{\alpha}{2}}$ on $\mathbb{R}^d$ and let $r$ be defined as $\frac{1}{r}=\frac{1}{p}-\frac{\alpha}{2d}$, where $p$ and $r$ are related by the Hardy-Littlewood-Sobolev fractional integration theorem. This theorem, applicable for $0<\alpha<2d$ and $1<p<r<\infty$, establishes the existence of a constant $C$ (dependent solely on $p$) such that $\| I_{\sfrac{\alpha}{2}} f \|_{L^r} \leq C\| f \|_{L^p}$ is satisfied for suitable functions.  
Then, for $f=\Lambda^{\frac{\alpha}{2}} b$ and $p=2$, we derive the inequality $\displaystyle \| b\|_{L^{\frac{2d}{d-\alpha}}} \leq C\| \Lambda^{\frac{\alpha}{2}} b \|_{L^2}$ for $\alpha \neq d$. 
Next, from interpolation inequality for $L^p$ norms, we see that $\displaystyle  \|b\|_{L^{2}}^{1+\frac{\alpha}{d}} \leq \|b\|_{L^{1}}^{\frac{\alpha}{d}} \| b \|_{L^{\frac{2d}{d-\alpha}}}$. 
%
Therefore,
\begin{equation}    
    \label{eq-2.3-k-2}
    \| b \|_{L^2}^{1+\frac{\alpha}{d}} \leq C \| b\|_{L^1}^{\frac{\alpha}{d}} \| \Lambda^{\frac{\alpha}{2}} b \|_{L^2}.
\end{equation}
Note that, for $\alpha=2$, we can use Nash's inequality, which yields \eqref{eq-2.3-k-2} directly: $\|b\|_{L^2}^{1+\frac{2}{d}} \leq C\|b\|_{L^1}^{\frac{2}{d}}\|\nabla b\|_{L^2}$. Therefore, for $\alpha=d=2$, \eqref{eq-2.3-k-2} still follows. 

Now, rewriting \eqref{eq-Lemma-6.3-1} with inequality \eqref{eq-2.3-k-2}, we obtain 
\begin{equation*}
    \frac{1}{2} \frac{\mathrm{d}}{\mathrm{d} t}\|b\|_{L^{2}}^{2} \leq -C \frac{\|b\|_{L^{2}}^{2+\frac{2\alpha}{d}}}{\|b\|_{L^{1}}^{\frac{2\alpha}{d}}}=-C \frac{\|b\|_{L^{2}}^{2+\frac{2\alpha}{d}}}{\|b_{0}\|_{L^{1}}^{\frac{2\alpha}{d}}}, 
\end{equation*}
with the final step obtained from the conservation of the $L^{1}$ norm of $b$. Then, setting $z(t)=\|b(\cdot, t)\|_{L^{2}}^{2}$, we have
\begin{equation*}
    z^{\prime}(t) \leq-C z(t)^{1+\frac{\alpha}{d}}\|b_{0}\|_{L^{1}}^{-\frac{2\alpha}{d}},
\end{equation*}
and solving this differential inequality, we find that
\begin{equation*}
    z(t) \leq\left(\frac{\alpha C t}{d\|b_{0}\|_{L^{1}}^{2 \alpha / d}}+\frac{1}{\|b_{0}\|_{L^{2}}^{2 \alpha / d}}\right)^{-d / \alpha},
\end{equation*}
which implies 
the first inequality in \eqref{eq-2.3-k}.

The second inequality in \eqref{eq-2.3-k} emerges via a duality argument using the incompressibility of the velocity field $u$. This approach involves considering $\theta(x, s)$, which is a solution to 
\begin{equation*}
    \partial_{s} \theta+u(x, t-s) \cdot \nabla \theta=-\Lambda^\alpha \theta, \quad \theta(x, 0)=\theta_{0}(x) \in \mathcal{S}.
\end{equation*} 
Note that $\theta(x, s)$ also satisfies the first estimate in \eqref{eq-2.3-k}, and we have 
\begin{equation}
    \label{eq-2.4-k}
    \frac{d}{d s} \int_{\mathbb{R}^{d}} b(x, s) \theta(x, t-s) \mathrm{~d} x=0,
\end{equation} 
which stems from the self-adjoint property of $\Lambda^\alpha$ leading to the equivalence of integrals
\begin{equation*}
    \int_{\mathbb{R}^{d}} \Lambda^\alpha b(x, s)  \theta(x, t-s) \mathrm{~d} x=\int_{\mathbb{R}^{d}} b(x, s) \Lambda^\alpha \theta(x, t-s) \mathrm{~d} x.
\end{equation*} 

Then, considering first $s=t$ and then $s=t/2$ on \eqref{eq-2.4-k}, we obtain 
\begin{equation*}
    \begin{split}
        \|b(\cdot, t)\|_{L^{\infty}}&=\sup _{\substack{\|\theta_{0}\|_{L^{1}} = 1 }}\left|\int_{\mathbb{R}^{d}} b(x,t) \theta_{0}(x)  \mathrm{~d} x\right|\\ 
        &=\sup _{\substack{\|\theta_{0}\|_{L^{1}} = 1 }}\left|\int_{\mathbb{R}^{d}}  b(x,t/2) \theta(x,t/2) \mathrm{~d} x\right|\\
        & \leq \sup _{\substack{\|\theta_{0}\|_{L^{1}} = 1 }} \|b(\cdot, t/2)\|_{L^{2}} \|\theta(\cdot, t/2)\|_{L^{2}}\\ 
         & \leq \sup _{\substack{\|\theta_{0}\|_{L^{1}} = 1 }} C (t/2)^{-d / 2\alpha} \|b_{0}\|_{L^{1}}  (t/2)^{-d / 2\alpha}  \|\theta_{0}\|_{L^{1}}\\
         & \leq C(d,\alpha) \ t^{-d / \alpha}\|b_{0}\|_{L^{1}}.
    \end{split}    
\end{equation*}
Here we used the first inequality in \eqref{eq-2.3-k} and adjusted $C(d,\alpha)$. Finally, from the weak parabolic maximum principle, we have $\|b(\cdot, t)\|_{L^{\infty}} \leq \|b_{0}\|_{L^{\infty}}$ \citep{User-guide-fractional-Laplacian,New-developments-nonlocal-operators-I,Hopf-lemmas-for-parabolic-fractional-Laplacians}. 
\end{proof}

\begin{lemma} \label{Lemma-2.2-k}
    Assume that $\rho$ solves \eqref{eq-2.1-k} with a smooth, bounded, and divergence-free $u$, and $\rho_{0} \in \mathcal{S}$. Then, for every $t>0$, we have
    \begin{equation*}
        \frac{\|\rho(\cdot, t)\|_{L^{p}}}{\|\rho(\cdot, t)\|_{L^{1}}} \leq \frac{\|\rho_{0}\|_{L^{p}}}{\|\rho_{0}\|_{L^{1}}} \; \; \text{ for all } \; \; 1 \leq p \leq \infty.
    \end{equation*}
\end{lemma}
\begin{proof}
    Here, for the sake of completeness, we provide a proof similar to the one presented by \citet{Kiselev-Biomixing}, but considering the Fractional Laplacian. 
    To start, note that for $p=1$ the result is immediate. Moving on to the case of $1<p<\infty$, we use \eqref{monotone-decreasing}, that is, $\displaystyle \frac{\mathrm{d}}{\mathrm{d} t}\|\rho(\cdot, t)\|_{L^{1}}=- \epsilon \|\rho(\cdot, t)\|_{L^{q}}^q$, to find
    \begin{equation*}
        \|\rho(\cdot, t)\|_{L^{1}}^{2} \frac{\mathrm{d}}{\mathrm{d} t} \left(\frac{\|\rho(\cdot, t)\|_{L^{p}}}{\|\rho(\cdot, t)\|_{L^{1}}}\right)
        =  \|\rho(\cdot, t)\|_{L^{1}} \frac{\mathrm{d}}{\mathrm{d} t}\|\rho(\cdot, t)\|_{L^{p}} +\epsilon \|\rho(\cdot, t)\|_{L^{q}}^q \|\rho(\cdot, t)\|_{L^{p}}, 
    \end{equation*}
    and from \Cref{Cordoba_Cordoba}, we have
    \begin{equation}
        \label{eq-Lemma-2.2-k-0}
        \begin{split}
            \frac{\mathrm{d}}{\mathrm{d} t}\|\rho(\cdot, t)\|_{L^{p}}
             & = \|\rho(\cdot, t)\|_{L^{p}}^{1-p} \int_{\mathbb{R}^d} \rho^{p-1}\left(-u \cdot \nabla \rho-\Lambda^{\alpha} \rho-\epsilon \rho^q\right) \mathrm{~d} x \\
             &\leq \|\rho(\cdot, t)\|_{L^{p}}^{1-p} \left( -\frac{2}{p}\| \Lambda^{\alpha/2} \rho^{p/2} \|_{L^2}^2 -\epsilon \int_{\mathbb{R}^d} \rho^{q+p-1} \mathrm{~d} x\right).
        \end{split}
    \end{equation}
    Therefore,
    \begin{multline}
        \label{eq-Lemma-2.2-k}
        \|\rho(\cdot, t)\|_{L^{1}}^{2} \frac{\mathrm{d}}{\mathrm{d} t} \frac{\|\rho(\cdot, t)\|_{L^{p}}}{\|\rho(\cdot, t)\|_{L^{1}}}
        \leq \|\rho(\cdot, t)\|_{L^{p}}^{1-p} \left[ -\frac{2}{p}\|\rho(\cdot, t)\|_{L^{1}} \| \Lambda^{\alpha/2} \rho^{p/2} \|_{L^2}^2  \right. \\
        +\left.\epsilon \left(\|\rho(\cdot, t)\|_{L^{q}}^q \|\rho(\cdot, t)\|_{L^{p}}^p-\|\rho(\cdot, t)\|_{L^{1}} \int_{\mathbb{R}^d} \rho^{q+p-1} \mathrm{~d} x \right) \right].
    \end{multline}

    Now, by applying Hölder's inequality, we can observe that the expression
    
    \begin{equation*}
        \int_{\mathbb{R}^d} \rho^q \mathrm{~d} x \int_{\mathbb{R}^d} \rho^p \mathrm{~d} x -\int_{\mathbb{R}^d} \rho \mathrm{~d} x \int_{\mathbb{R}^d} \rho^{q+p-1} \mathrm{~d} x 
    \end{equation*}
    is less than or equal to zero. The case $p=\infty$ follows from a limiting procedure, since $\rho \in \mathcal{S}$ for all $t$.
\end{proof}
\vspace{0.1cm}

Now we prove \Cref{Theorem-1.1-k}. 

\vspace{0.1cm}
\begin{proof}[Proof of \Cref{Theorem-1.1-k}]
    As in \cite{Kiselev-Biomixing}, the idea here is to show that if the $L^{1}$ norm of $\rho$ at some time $t_{0}$ is sufficiently small then, for all times $t>t_{0}$, the $L^{1}$ norm of $\rho$ cannot drop below $\|\rho(\cdot, t_{0})\|_{L^{1}} / 2$. This shows that $\rho(x, t)$ cannot tend to zero as $t \rightarrow+\infty$.
    For this, recall that, for every $t$,
    \begin{equation*}
        \frac{\mathrm{d}}{\mathrm{d} t} \int_{\mathbb{R}^{d}} \rho(x, t) \mathrm{~d} x=-\epsilon \int_{\mathbb{R}^{d}} \rho(x, t)^{q} \mathrm{~d} x \geq-\epsilon \int_{\mathbb{R}^{d}} b(x, t)^{q} \mathrm{~d} x,
    \end{equation*}
    where $b$ is given by \eqref{eq-2.2-k}. From \Cref{lemma-2.1-k} and Hölder's inequality,
    \begin{equation*}
        \int_{\mathbb{R}^{d}} b(x, t)^{q} \mathrm{~d} x 
        \leq C \min \left(\|\rho_{0}\|_{L^{\infty}}^{q-1} \|\rho_{0}\|_{L^{1}}, t^{-\frac{d(q-1)}{\alpha}}\|\rho_{0}\|_{L^{1}}^{q}\right). 
    \end{equation*}

    Thus, for every $\tau>0$,
    \begin{equation*}
        \begin{aligned}
            \int_{t_{0}}^{\infty} \int_{\mathbb{R}^{d}} b(x, t)^{q} \mathrm{~d} x \mathrm{d} t & =\int_{t_{0}}^{t_{0}+\tau} \int_{\mathbb{R}^{d}} b(x, t)^{q} \mathrm{~d} x \mathrm{d} t +\int_{t_{0}+\tau}^{\infty} \int_{\mathbb{R}^{d}} b(x, t)^{q} \mathrm{~d} x \mathrm{d} t\\
            & \leq C(d, \alpha)\left(\|\rho(\cdot, t_{0})\|_{L^{\infty}}^{q-1}\|\rho(\cdot, t_{0})\|_{L^{1}} \tau+\|\rho(\cdot, t_{0})\|_{L^{1}}^{q} \int_{t_{0}+\tau}^{\infty}\left(t-t_{0}\right)^{-\frac{d(q-1)}{\alpha}} \mathrm{d} t\right),
        \end{aligned}
    \end{equation*}
    and, as by assumption $q d>d+\alpha$, we have
    \begin{equation}
        \label{eq-2.5-k}
        \begin{aligned}
            \int_{t_{0}}^{\infty} \int_{\mathbb{R}^{d}} b(x, t)^{q} \mathrm{~d} x \mathrm{d} t \leq C(d, \alpha, q)\left(\|\rho(\cdot, t_{0})\|_{L^{\infty}}^{q-1}\|\rho(\cdot, t_{0})\|_{L^{1}} \tau+\|\rho(\cdot, t_{0})\|_{L^{1}}^{q} \tau^{\frac{d+\alpha-q d}{\alpha}}\right).
        \end{aligned}
    \end{equation}

    Assume, on the contrary, that the $L^{1}$ norm of $\rho$ does go to zero for some $u$ and $\rho_0$. Then, consider some time $t_{0}>0$ when $\|\rho(\cdot, t_{0})\|_{L^{1}}$ is sufficiently small. By \Cref{Lemma-2.2-k} and \eqref{eq-2.5-k}, we see that further decrease of the $L^{1}$ norm from that level is bounded as 
    \begin{equation}
        \label{eq-2.6-k}
        \|\rho(\cdot, t_{0})\|_{L^{1}}-\|\rho(\cdot, t)\|_{L^{1}} \leq C(d, \alpha, q) \epsilon\left(\frac{\|\rho_{0}\|_{L^{\infty}}^{q-1}}{\|\rho_{0}\|_{L^{1}}^{q-1}}\|\rho(\cdot, t_{0})\|_{L^{1}}^{q} \tau+\|\rho(\cdot, t_{0})\|_{L^{1}}^{q} \tau^{\frac{d+\alpha-q d}{\alpha}}\right)
    \end{equation}
    for all $t>t_{0}$, $\tau>0$. Choosing $\tau$ to minimize expression \eqref{eq-2.6-k}, for every $t>t_{0}$, we find that
    \begin{equation*}
        \|\rho(\cdot, t_{0})\|_{L^{1}}-\|\rho(\cdot, t)\|_{L^{1}} \leq C(d, \alpha, q) \epsilon\|\rho(\cdot, t_{0})\|_{L^{1}}^{q}\left(\frac{\|\rho_{0}\|_{L^{\infty}}}{\|\rho_{0}\|_{L^{1}}}\right)^{\frac{q d-d-\alpha}{d}}.
    \end{equation*}

    If $\, \|\rho(\cdot, t)\|_{L^{1}} \rightarrow 0$ as $t \rightarrow+\infty$, we may choose $t_{0}$ so that
    \begin{equation}
        \label{eq-2.8-k}
        C(d, \alpha, q) \epsilon\|\rho(\cdot, t_{0})\|_{L^{1}}^{q-1}\left(\frac{\|\rho_{0}\|_{L^{\infty}}}{\|\rho_{0}\|_{L^{1}}}\right)^{\frac{q d-d-\alpha}{d}} \leq \frac{1}{2}.
    \end{equation}

    Then we get a contradiction to the assumption that $\|\rho(\cdot, t)\|_{L^1} \rightarrow 0$ as $t \rightarrow +\infty$, as 
    \begin{equation*}
        \|\rho(\cdot, t)\|_{L^1} \geq \frac{1}{2}\|\rho(\cdot, t_{0})\|_{L^1} \geq C_0\left(q, d,\epsilon,\alpha, \rho_0\right)
    \end{equation*}
    for every $t>t_0$, where $C_0$ in the statement of the theorem can be defined as
    \begin{equation*}
        C_0\left(q, d,\epsilon,\alpha, \rho_0\right) \equiv \min \left(\frac{1}{2}\|\rho_0\|_{L^1}, \frac{1}{2^{\frac{q}{q-1}} \epsilon^{\frac{1}{q-1}} C(q, d, \alpha)^{\frac{1}{q-1}}}\left(\frac{\|\rho_0\|_{L^1}}{\|\rho_0\|_{L^{\infty}}}\right)^{1-\frac{\alpha}{d(q-1)}}\right).
    \end{equation*}
    
    Note that, if $\epsilon \rightarrow 0$ while $\rho_{0}$, $u$ and $q$ are fixed, we can replace condition $\leq \frac{1}{2}$ in \eqref{eq-2.8-k} with $\leq \kappa$, where $\kappa$ can be taken as small as desired, proving the last statement of the theorem.
\end{proof}

\subsection{Reaction in a Chemotactic Environment} 
\label{Sec:Reaction-Enhancement-Chemotaxis}

In the chemotactic environment, we prove that the large time limit of the $L^{1}$ norm of $\rho$ tends to zero as chemotaxis coupling increases, with an upper bound independent of $\epsilon$ (\Cref{Theorem-4.2-k}). 
However, we establish lower bounds for the $L^{1}$ norm of the solution, showing that, for each fixed coupling, $\|\rho(\cdot, t)\|_{L^{1}}$ does not go to zero as $t \rightarrow \infty$, as outlined next.

\subsubsection*{Lower bound for the total fraction of unfertilized eggs} \label{Subsec:Lower-bound-fraction-for-m(t)}
\addcontentsline{toc}{subsubsection}{\ \ \ \ Lower bound for the total fraction of unfertilized eggs}

We proceed by setting a lower bound for the $L^1$ norm of $\rho$. 

\begin{theorem}\label{Theorem-4.3-k}  
    Let $d=2$, $q, s>2$ be integers, $\alpha \in (1,2)$, and $\beta$ satisfy $0\leq \beta < \alpha$.   Suppose that $u \in C^{\infty}(\mathbb{R}^d \times[0, \infty))$ is divergence free,  and $\rho$ solves \eqref{eq-4.1-k} with $\rho_0 \geq 0 \in K_{s, \beta}$. 
    Then,  $\underset{t \rightarrow \infty}{\lim}\|\rho(\cdot, t)\|_{L^{1}}>0$.
\end{theorem}
\begin{proof}
In order to establish lower bounds on the $L^1$ norm of the solution $\rho$, let us deduce estimates on $\|\rho(\cdot,t)\|_{L^{q}}$. By multiplying \eqref{eq-4.1-k} by $\rho^{q-1}$ and integrating, we obtain
 \begin{equation}
     \label{eq-4.9-k}
     \frac{1}{q} \frac{\mathrm{d} }{\mathrm{d} t} \int_{\mathbb{R}^{2}} \rho^{q} \mathrm{d}  x=-\int_{\mathbb{R}^{2}} \rho^{q-1} \Lambda^{\alpha} \rho \mathrm{d}  x+\chi \int_{\mathbb{R}^{2}} \rho^{q-1} \nabla \cdot\left(\rho \nabla \Delta^{-1} \rho\right) \mathrm{d}  x-\epsilon \int_{\mathbb{R}^{2}} \rho^{2 q-1} \mathrm{d}  x.
 \end{equation}

Note that, from \Cref{Cordoba_Cordoba}, we have
\begin{equation}
    \label{Ineq-Th-4.3-01}
    -\int_{\mathbb{R}^2} \rho^{q-1} \Lambda^{\alpha} \rho \mathrm{~d} x \leq -\frac{2}{q}\| \Lambda^{\alpha/2} \rho^{q/2} \|_{L^2}^2,
\end{equation}
and using integration by parts, we find
\begin{equation}
    \label{Ineq-Th-4.3-02}
    \int_{\mathbb{R}^{2}} \rho^{q-1} \nabla \cdot\left(\rho \nabla \Delta^{-1} \rho\right)\mathrm{~d} x=-(q-1) \int_{\mathbb{R}^{2}} \rho^{q-1} \nabla \rho \cdot \nabla \Delta^{-1} \rho\mathrm{~d} x=\frac{q-1}{q} \int_{\mathbb{R}^{2}} \rho^{q+1}\mathrm{~d} x.
\end{equation}
Furthermore, from interpolation inequality for $L^p$ norms, we see that
\begin{equation} 
    \label{Ineq-Th-4.3-03}
    \begin{split}
        \int_{\mathbb{R}^{2}} \rho^{q+1}\mathrm{~d} x=\|\rho^{q / 2}\|_{L^{2+\frac{2}{q}}}^{\frac{2(q+1)}{q}} &\leq 
        C(q, \alpha)\| \rho^{q / 2} \|_{L^{\frac{4}{2-\alpha}}}^{\frac{4q}{2q-2+\alpha}} \|\rho^{q / 2}\|_{L^{\frac{2}{q}}}^{\frac{2}{q}\left(\frac{\alpha q-2+\alpha}{2q-2+\alpha}\right)}.
    \end{split}
\end{equation}

Now, employing Hardy-Littlewood-Sobolev fractional integration theorem, which states that, for $0<\upsilon<d$, $\| \Lambda^{-\upsilon} f \|_{L^r} \leq C\| f \|_{L^p}$,  where $C$ is a constant depending only on $p$, and $p$ and $r$ satisfy $\frac{1}{r}=\frac{1}{p}-\frac{\upsilon}{d}$ with $1<p<r<\infty$, we can derive that
\begin{equation}
    \label{Ineq-Th-4.3-04}
    \| \rho^{q / 2} \|_{L^{\frac{4}{2-\alpha}}} \leq C\| \Lambda^{\frac{\alpha}{2}} \rho^{q / 2} \|_{L^2}.
\end{equation} 
Moreover, by standard interpolation inequality for $L^p$ norms, we obtain 
\begin{equation}
    \label{Ineq-Th-4.3-05}
    \| \rho^{q / 2} \|_{L^{\frac{4}{2-\alpha}}}^{\frac{2}{q}}=\| \rho \|_{L^{\frac{2q}{2-\alpha}}} \leq \|\rho\|_{L^{1}}^{\frac{2-\alpha}{2q}} \| \rho \|_{L^{\infty}}^{\frac{2q-2+\alpha}{2q}} .
\end{equation}

Then, from \eqref{Ineq-Th-4.3-04} and \eqref{Ineq-Th-4.3-05}, we establish that
\begin{equation*}
    \| \rho^{q / 2} \|_{L^{\frac{4}{2-\alpha}}}^{\frac{4q}{2q-2+\alpha}} = \| \rho^{q / 2} \|_{L^{\frac{4}{2-\alpha}}}^{2} \| \rho^{q / 2} \|_{L^{\frac{4}{2-\alpha}}}^{\frac{2(2-\alpha)}{2q-2+\alpha}} \leq C \| \Lambda^{\frac{\alpha}{2}} \rho^{q / 2} \|_{L^2}^2 \|\rho\|_{L^{1}}^{\frac{(2-\alpha)^2}{2(2q-2+\alpha)}} \| \rho \|_{L^{\infty}}^{\frac{2-\alpha}{2}}. 
\end{equation*}
Applying this to \eqref{Ineq-Th-4.3-03}, we can write
\begin{equation*} 
    \begin{split}
        \int_{\mathbb{R}^{2}} \rho^{q+1}\mathrm{~d} x &\leq C(q,\alpha) \| \Lambda^{\frac{\alpha}{2}} \rho^{q / 2} \|_{L^2}^2 \|\rho\|_{L^{1}}^{\frac{(2-\alpha)^2}{2(2q-2+\alpha)}} \| \rho \|_{L^{\infty}}^{\frac{2-\alpha}{2}} \|\rho^{q / 2}\|_{L^{\frac{2}{q}}}^{\frac{2}{q}\left(\frac{\alpha q-2+\alpha}{2q-2+\alpha}\right)}\\
        &\leq C(q,\alpha) \| \Lambda^{\frac{\alpha}{2}} \rho^{q / 2} \|_{L^2}^2 \|\rho\|_{L^{1}}^{\frac{(2-\alpha)^2}{2(2q-2+\alpha)}+\frac{\alpha q-2+\alpha}{2q-2+\alpha}} \| \rho \|_{L^{\infty}}^{\frac{2-\alpha}{2}}.
    \end{split}    
\end{equation*}

Thus, from \eqref{Ineq-Th-4.3-01}, \eqref{Ineq-Th-4.3-02}, and the inequality above, we obtain 
\begin{equation*}
    \begin{split}
        & q\left( -\int_{\mathbb{R}^{2}} \rho^{q-1} \Lambda^{\alpha} \rho \mathrm{d}  x+\chi \int_{\mathbb{R}^{2}} \rho^{q-1} \nabla \cdot\left(\rho \nabla \Delta^{-1} \rho\right) \mathrm{d}  x \right)  \leq  \\
        & \hspace{5cm} \| \Lambda^{\frac{\alpha}{2}} \rho^{q/2} \|_{L^2}^2 \left( C(q,\alpha) \chi \|\rho\|_{L^{1}}^{\frac{\alpha}{2} \left(1+\frac{2}{2q-2+\alpha} \right)} \| \rho \|_{L^{\infty}}^{\frac{2-\alpha}{2}} -2 \right).
    \end{split}
\end{equation*}
Therefore, we can rewrite \eqref{eq-4.9-k} as 
\begin{equation}
     \label{eq-4.9-k-2}
    \frac{\mathrm{d} }{\mathrm{d} t} \int_{\mathbb{R}^{2}} \rho^{q} \mathrm{d}  x
    \leq   N_0^{\frac{2-\alpha}{2}} \| \Lambda^{\frac{\alpha}{2}} \rho^{q/2} \|_{L^2}^2 \left( C(q,\alpha) \chi \|\rho\|_{L^{1}}^{\frac{\alpha}{2} \left(1+\frac{2}{2q-2+\alpha} \right)}  - 2 N_0^{-\frac{2-\alpha}{2}} \right)-q\epsilon \int_{\mathbb{R}^{2}} \rho^{2 q-1} \mathrm{d}  x,
 \end{equation}
where $N_0$ is a uniform in time upper bound of $\|\rho(\cdot, t)\|_{L^{\infty}}$, \ie $\|\rho\|_{L^{\infty}\left(\mathbb{R}_{+} ; L^{\infty}\right)}:= \underset{t \in [0,\infty)}{\operatorname{ess} \sup}\|\rho(\cdot, t)\|_{L^{\infty}} \leq N_0$. 

Note that $N_0=\max \left((\chi / \epsilon)^{\frac{1}{q-2}},\|\rho_0\|_{L^{\infty}}\right) <\infty$ according to \Cref{Lemma-5.6-k}, and $\|\rho(\cdot,t)\|_{L^{1}}$ is non-increasing in time. 
Now supposing that, at some time $t_{0}$, $C(q,\alpha) \chi \|\rho\|_{L^{1}}^{\frac{\alpha}{2} \left(1+\frac{2}{2q-2+\alpha} \right)}$ drops below $2 N_0^{-\frac{2-\alpha}{2}}$, we also have
\begin{equation}
    \label{eq-assumption}
    C(q,\alpha) \chi \|\rho(\cdot,t)\|_{L^{1}}^{\frac{\alpha}{2} \left(1+\frac{2}{2q-2+\alpha} \right)} < 2 N_0^{-\frac{2-\alpha}{2}}, \qquad \forall \ t>t_{0}.
\end{equation}
Then, for all later times, we get
 \begin{equation}
    \label{eq-4.12-k}
    \frac{\mathrm{d} }{\mathrm{d} t} \int_{\mathbb{R}^{2}} \rho^{q} \mathrm{d}  x \leq  - N_0^{\frac{2-\alpha}{2}} \| \Lambda^{\frac{\alpha}{2}} \rho^{q/2} \|_{L^2}^2,  
\end{equation}
and, again from interpolation inequality for $L^p$ norms, we obtain 
\begin{equation*}
    \| \rho^{q/2} \|_{L^{2}} \leq C(q,\alpha) \| \rho^{q/2} \|_{L^{\frac{4}{2-\alpha}}}^{\frac{2(q-1)}{2q-2+\alpha}}\|\rho^{q/2}\|_{L^{\frac{2}{q}}}^{\frac{\alpha}{2q-2+\alpha}}.
\end{equation*}
Hence, from this and \eqref{Ineq-Th-4.3-04}, we have 
\begin{equation*}
    \| \rho^{q/2} \|_{L^{2}}^{1+\frac{\alpha}{2(q-1)}}  \leq C(q,\alpha) \| \Lambda^{\frac{\alpha}{2}} \rho^{q/2} \|_{L^2} \|\rho^{q/2}\|_{L^{2/q}}^{\frac{\alpha}{2(q-1)}},
\end{equation*}
which, applied to \eqref{eq-4.12-k}, leads to
\begin{equation}
    \label{eq-Theorem-4.3-k-1-1}
    \frac{\mathrm{d} }{\mathrm{d} t} \int_{\mathbb{R}^{2}} \rho^{q} \mathrm{~d}  x \leq  - C(q,\alpha) \ N_0^{\frac{2-\alpha}{2}} \left(\int_{\mathbb{R}^{2}} \rho^{q} \mathrm{~d} x\right)^{1+\frac{\alpha}{2}\left(\frac{1}{q-1}\right)}\left(\int_{\mathbb{R}^{2}} \rho \mathrm{~d} x\right)^{-\frac{\alpha}{2}\left(\frac{q}{q-1}\right)}.
\end{equation}

Then, using the fact that $\displaystyle \int_{\mathbb{R}^{2}} \rho(x, t) \mathrm{~d} x$ is monotone decreasing and introducing $z(t)=\|\rho(\cdot, t)\|_{L^{q}}^{q}$,  we can represent \eqref{eq-Theorem-4.3-k-1-1} as
\begin{equation*}
    z^{\prime}(t) \leq -C(q,\alpha)  \ N_0^{\frac{2-\alpha}{2}}  z(t)^{1+\frac{\alpha}{2}\left(\frac{1}{q-1}\right)} \|\rho(\cdot,t_{*})\|_{L^{1}}^{-\frac{\alpha}{2}\left(\frac{q}{q-1}\right)}
\end{equation*}
for $t_{*}\geq t_{0}$.
Solving this differential inequality, we find 
\begin{equation*}
    z(t) \leq\left(\frac{C(q,\alpha) \ N_0^{\frac{2-\alpha}{2}} (t-t_{*})}{\|\rho(\cdot,t_{*})\|_{L^{1}}^{\frac{\alpha}{2}\left(\frac{q}{q-1}\right)}}+\frac{1}{\|\rho(\cdot, t_{*})\|_{L^{q}}^{\frac{\alpha}{2}\left(\frac{q}{q-1}\right)}}\right)^{-\frac{2}{\alpha}\left(q-1\right)},
\end{equation*}
implying
\begin{equation}
    \label{eq-Theorem-4.3-k-1}
    \|\rho(\cdot, t)\|_{L^{q}}^{q} \leq \min \left(\|\rho(\cdot, t_{*})\|_{L^{q}}^{q}, C(q,\alpha) \ N_0^{-\frac{2-\alpha}{\alpha}(q-1)} (t-t_{*})^{-\frac{2}{\alpha}\left(q-1\right)} \|\rho(\cdot, t_{*})\|_{L^{1}}^{q}\right).
\end{equation}

Next, setting $p=q$ and $t>t_0$ as in the proof of \Cref{Lemma-2.2-k}, we can certify that 
\begin{equation}
    \label{eq-Theorem-4.3-k-2}
    \frac{\|\rho(\cdot, t)\|_{L^{q}}}{\|\rho(\cdot, t)\|_{L^{1}}} \leq \frac{\|\rho(\cdot, t_{0})\|_{L^{q}}}{\|\rho(\cdot, t_{0})\|_{L^{1}}}. 
\end{equation}
Indeed, with chemotaxis, an inequality equivalent to  \eqref{eq-Lemma-2.2-k-0} can be deduced from \eqref{eq-4.9-k-2}. 
This leads to an inequality equivalent to \eqref{eq-Lemma-2.2-k} whose expression inside square brackets turns out to be 
\begin{multline*}
    q^{-1}\|\rho(\cdot, t)\|_{L^{1}} \| \Lambda^{\frac{\alpha}{2}} \rho^{q/2}(\cdot, t) \|_{L^2}^2 \left( C(q,\alpha) \chi \|\rho(\cdot, t)\|_{L^{1}}^{\frac{\alpha}{2} \left(1+\frac{2}{2q-2+\alpha} \right)} \| \rho (\cdot, t)\|_{L^{\infty}}^{\frac{2-\alpha}{2}} -2 \right) \\+\epsilon \left(\|\rho(\cdot, t)\|_{L^{q}}^{2q} -\|\rho(\cdot, t)\|_{L^{1}} \int_{\mathbb{R}^d} \rho^{2q-1} \mathrm{~d} x \right), 
\end{multline*}
which, based on assumption \eqref{eq-assumption} and Hölder’s inequality, is negative for all $t>t_0$.  

Now, for every $\tau>0$ and $t_{*}\geq t_0$, we obtain
\begin{equation*}
    \begin{aligned}
        \int_{t_{*}}^{\infty} \int_{\mathbb{R}^{d}} \rho(x,t)^{q} \mathrm{~d} x \mathrm{d} t & =\int_{t_{*}}^{t_{*}+\tau} \int_{\mathbb{R}^{d}} \rho(x,t)^{q} \mathrm{~d} x \mathrm{d} t +\int_{t_{*}+\tau}^{\infty} \int_{\mathbb{R}^{d}} \rho(x,t)^{q} \mathrm{~d} x \mathrm{d} t\\
        \leq C(q,\alpha) & \left(\|\rho(\cdot, t_{*})\|_{L^{q}}^q \tau+ N_0^{-\frac{2-\alpha}{\alpha}(q-1)} \|\rho(\cdot, t_{*})\|_{L^{1}}^{q} \int_{t_{*}+\tau}^{\infty}\left(t-t_{*}\right)^{-\frac{2(q-1)}{\alpha}} \mathrm{d} t\right)\\
        \leq C(q,\alpha) & \left(\|\rho(\cdot, t_{*})\|_{L^{q}}^q \tau+ N_0^{-\frac{2-\alpha}{\alpha}(q-1)} \|\rho(\cdot, t_{*})\|_{L^{1}}^{q} \tau^{\frac{2+\alpha-2q}{\alpha}}\right),
    \end{aligned}
\end{equation*}
where we use \eqref{eq-Theorem-4.3-k-1} and the fact that  $q >1+\frac{\alpha}{2}$. Thus, using \eqref{eq-Theorem-4.3-k-2} and Hölder’s inequality, we see that
\begin{equation}
    \label{eq-Theorem-4.3-k-3}
    \|\rho(\cdot, t_{*})\|_{L^1}-\|\rho(\cdot, t)\|_{L^1} \leq C(q,\alpha) \epsilon \|\rho(\cdot, t_{*})\|_{L^1}^q \left(\frac{\|\rho(\cdot, t_{0})\|_{L^{\infty}}^{q-1}}{\|\rho(\cdot, t_{0})\|_{L^1}^{q-1}} \tau+N_0^{-\frac{2-\alpha}{\alpha}(q-1)} \tau^{\frac{2+\alpha-2q}{\alpha}}\right)
\end{equation}
for all $t>t_*>t_0$, and $\tau>0$.
%
Subsequently, choosing $\tau$ to minimize expression \eqref{eq-Theorem-4.3-k-3}, we find, for every $t>t_*>t_0$, that 
\begin{equation*}
    \|\rho(\cdot, t_{*})\|_{L^1}-\|\rho(\cdot, t)\|_{L^1} \leq C(q,\alpha) \epsilon  N_0^{-\frac{2-\alpha}{2}}  \|\rho(\cdot, t_{*})\|_{L^1}^q 
    \left(\frac{\|\rho(\cdot, t_{0})\|_{L^{\infty}}}{\|\rho(\cdot, t_{0})\|_{L^1}}\right)^{\frac{2(q-1)-\alpha}{2}} .
\end{equation*}

Then, employing the same argument used in the proof of \Cref{Theorem-1.1-k}, we obtain
\begin{equation}
    \label{eq-4.14-k}
    \inf _{t} \int_{\mathbb{R}^{2}} \rho(x, t) \mathrm{~d} x \geq \min \left(\frac{1}{2}\|\rho(\cdot, t_{0})\|_{L^{1}}, C(q,\alpha) \ N_0^{\frac{2-\alpha}{2\left(q-1\right)}} \epsilon^{-\frac{1}{q-1}}\left(\frac{\|\rho(\cdot, t_{0})\|_{L^{1}}}{\|\rho(\cdot, t_{0})\|_{L^{\infty}}}\right)^{\frac{q-(1+\alpha/2)}{q-1}}\right).
\end{equation}

Observe that \eqref{eq-4.14-k} implies $\displaystyle \lim _{t \rightarrow \infty}\|\rho(\cdot, t)\|_{L^{1}}>0$, as global well-posedness of the solution has been established and, from \Cref{Lemma-5.6-k}, 
\begin{equation*}
    N_0^{\frac{2-\alpha}{2\left(q-1\right)}} \|\rho(\cdot, t_{0})\|_{L^{\infty}}^{-\frac{q-(1+\alpha/2)}{q-1}} \geq N_0^{-\frac{q-2}{q-1}}>0.
\end{equation*}
Additionally, if there exists $\rho$ such that \eqref{eq-4.14-k} is not satisfied, it implies that there is no $t>0$ for which \eqref{eq-assumption} holds. Consequently, once again, $\displaystyle \lim _{t \rightarrow \infty}\|\rho(\cdot, t)\|_{L^{1}}>0$.
%
\end{proof}
\subsubsection*{Upper bound for the total fraction of unfertilized eggs} \label{Subsec:Upper-bound-for-m(t)} 
\addcontentsline{toc}{subsubsection}{\ \ \ \  Upper bound for the total fraction of unfertilized eggs}

For the classical Keller-Segel equation (equation \eqref{eq-4.1-k}  with $\alpha=2$  and $\epsilon=u=0$) a virial argument has been employed to demonstrate the existence of blowing-up solutions. This argument involves studying the evolution of the second-order moment $m_2(t)=\int_{\mathbb{R}^d}|x-x_0|^2 \rho(x,t) \mathrm{d} x$, $x_0 \in \mathbb{R}^d$. Assuming the regularity of the solution, 
one can derive a differential inequality of the form 
\begin{equation}
    \label{eq-Upper-bound-1}
     \frac{\mathrm{d}}{\mathrm{d} t} m_2(t)\leq f(m(t),m_2(t)). 
\end{equation}
It can be shown that for $d \geq 2$ and large (in some sense) initial data, this inequality leads to negative second moment of the solution, leading to a contradiction since the initial data and hence the solution are positive. 
This implies that the solution must lose regularity before the time $m_2$ would turn negative
(see e.g.  \citep{book-base,Exploding-solutions-nonlocal-quadratic-evolution,Biler-fractional-9}). 

    For the classical Keller-Segel equation, 
    since  $m(t)$ remains constant over time ($m(t)=m_0$), for $x \in \mathbb{R}^d$, the function $f$ in \eqref{eq-Upper-bound-1} satisfies:
    \begin{equation}
        \label{eq-f-m0-m2-classical-Keller-Segel}
        f(m_0, m_2(t)) =  
        \begin{cases} 
            \displaystyle  4m_0\left(1 - \frac{\chi}{8\pi}m_0\right), & d = 2, \\
            \displaystyle  2m_0\left(-1 + C\frac{\chi m_2(t)}{(\chi m_0)^{\frac{d}{d-2}}}\right), & d > 2.
        \end{cases}  
    \end{equation}
    Thus, for $d=2$, if $m_0>8\pi/\chi$, $m_2(t)$ would become negative in finite time, which is impossible since $\rho$ is nonnegative. 
    Similarly, for $d > 2$, if $\chi m_2(0)$ is sufficiently small compared to $(\chi m_0)^{\frac{d}{d-2}}$, $m_2(t)$ decreases for all $t$, leading to a contradiction. 

\citet{Kiselev-Biomixing} applied this idea to establish an upper bound for $m(t)$ in equation  \eqref{eq-4.1-k} with $\alpha=2$ and $d=2,$ and integer $q>2$. They observed that when the reaction is present, solutions will remain globally regular. But now the analog 
of \eqref{eq-Upper-bound-1} implies that $m(t)$ must decay and quickly become small (at least if $\chi$ is large), or otherwise the second moment will turn negative. An interesting feature of their result is that the upper bound on $m(t)$ as well as the time scale 
when it is achieved are independent of $\epsilon.$ This is the ``shadow" of the singularity formation that takes place if $\epsilon=0:$ 
even if $\epsilon$ is very small, the dynamics will concentrate significant mass in a small region and the size of $\rho$ there 
will drive the reaction. 

The case $0<\alpha<2$ makes this argument inapplicable, as the second moment of a typical solution to an evolution equation with fractional Laplacian cannot be finite. \citet{Exploding-solutions-nonlocal-quadratic-evolution} pointed out that in this case the weight function $|x|^2$ makes the linear term too strong to be controlled by the nonlinear part. 
%
To overcome this limitation and prove the existence of blowing-up solutions for a nonlocal Keller-Segel equation (equation \eqref{eq-4.1-k}  with $1<\alpha<2$  and $\epsilon=u=0$), \citet{Biler-fractional-9} extended the classical method by studying moments of lower order $\gamma \in(1,2)$:
\begin{equation}
    m_\gamma(t)=\int_{\mathbb{R}^d}|x-x_0|^\gamma \rho(x,t) \mathrm{d} x , \quad x_0 \in \mathbb{R}^d.
\end{equation}

They emphasized that, for $\alpha<2$, the existence of higher-order moments $m_\gamma$ with $\gamma \geq \alpha$ cannot be expected. Even for the linear equation $\partial_t v+\left(-\Delta\right)^{\alpha / 2} v=0$, the fundamental solution $p_\alpha(x, t)$ behaves like $p_\alpha(x, t) \sim\left(t^{d / \alpha}+|x|^{d+\alpha} / t\right)^{-1}$, forcing moments with $\gamma \geq \alpha$ to be infinite.

\citet{Biler-fractional-9} then demonstrated the existence of blowing-up solutions by showing the finite-time extinction of the function (setting $x_0=0$ for simplicity)
 \begin{equation}
    \label{pseudo-moment}
     w(t) \equiv \int_{\mathbb{R}^d} \varphi(x) \rho(x,t) \mathrm{d} x,
 \end{equation}
 where $\varphi(x)$ is a smooth nonnegative weight function on $\mathbb{R}^d$ defined as
\begin{equation}
    \label{eq-4.1}
    \varphi(x) \equiv \left(1+|x|^2\right)^{\gamma / 2}-1,
\end{equation}
with $\gamma \in(1,2]$.  
\citet{Biler-fractional-9} highlighted that the quantity $w$ is essentially equivalent to the moment $m_\gamma$ of order $\gamma$ of the solution, as for every  $\varepsilon>0$, a suitably chosen $C(\varepsilon)>0$ ensures for every $x \in \mathbb{R}^d$ that
\begin{equation}
    \label{eq-4.2}
    \varphi(x) \leq|x|^\gamma \leq \varepsilon+C(\varepsilon) \varphi(x).
\end{equation} 

Drawing inspiration from the work of \citet{Kiselev-Biomixing} and \citet{Biler-fractional-9}, we analyze the ODE associated to the evolution of $w(t)$ to
provide an upper bound for $m(t)$ (\Cref{Theorem-4.2-k}) that remains unchanged despite variations in the coupling of the reaction term.  
As \Cref{Lemma-5.6-k} assures that the balance between chemotaxis and the reaction term (as far as $q>2$) leads to a smooth decaying solution, global regularity of the solutions is ensured, giving rise to \Cref{Theorem-4.1-k}. 
Consequently, the quantity $m_\gamma(t)$ cannot vanish,  even for initial condition $\|\rho_0\|_{L^1}$  sufficiently large and well-concentrated around a point ($m_\gamma(0)$ small enough), in contrast to what occurs when $\epsilon=0$, for $u=0$ (see \cite{Biler-fractional-9, Biler-fractional-2017} for details). It is important to point out that, differently from \cite{Kiselev-Biomixing}, our discussions are explicitly based on the sign of $\mathrm{d}w/\mathrm{d} t$ for the initial condition $(m_0,w_0)$.  
To start, consider the two following lemmas from \cite{Biler-fractional-9}:
\begin{lemma}{\citep{Biler-fractional-9}} \label{Lemma-4.1}
    Let $\alpha \in(1,2)$, $\gamma \in(1, \alpha)$, and $\varphi$ be defined as in \eqref{eq-4.1}. Then, $\left(-\Delta\right)^{\alpha / 2} \varphi \in L^{\infty}$.
\end{lemma}
\begin{lemma}{\citep{Biler-fractional-9}} \label{Lemma-4.2}
    For every $\gamma \in(1,2]$, the function $\varphi$ defined in \eqref{eq-4.1} is locally uniformly convex on $\mathbb{R}^{d}$. Moreover, there exists $K=K(\gamma)$ such that the inequality
    \begin{equation}
        \label{eq-4.10}
        (\nabla \varphi(x)-\nabla \varphi(y)) \cdot(x-y) \geq \frac{K|x-y|^{2}}{1+|x|^{2-\gamma}+|y|^{2-\gamma}}
    \end{equation}
    holds true for all $x, y \in \mathbb{R}^{d}$.
\end{lemma}
\begin{remark} \label{value-kgamma}
   The derivation of inequality \eqref{eq-4.10} hinges on the condition $\gamma>1$, and consequently $\alpha>1$, since $m_\gamma$ with $\gamma \geq \alpha$ cannot be expected to be finite. 
   Furthermore, it can be established that $K(\gamma)=\gamma-1$. 
\end{remark}
\begin{lemma} \label{Lemma-Inequality-Parte-1-Proof}
    Let $\alpha \in(1,2)$, $\gamma \in(1, \alpha)$, and $\varphi$ be defined as in \eqref{eq-4.1}. Then, the inequality
    \begin{equation}
        \label{eq-Inequality-Parte-1-Proof}
        \left|\frac{\nabla \varphi(x)}{\gamma}\right|^{\frac{\gamma}{\gamma-1}}\leq \frac{2}{\gamma}\varphi(x)
    \end{equation}
    holds for all $x, y \in \mathbb{R}^{d}$.
\end{lemma}
\begin{proof}
    Note that $\nabla \varphi(x)=\gamma\left(1+|x|^{2}\right)^{\frac{\gamma}{2}-1} x$ and both $|\nabla \varphi(x)|^{\frac{\gamma}{\gamma-1}}$ and $\varphi(x)$ are radially symmetric functions. Then, the functions $F, \; G:\mathbb{R}^d\backslash \{0\}\rightarrow \mathbb{R}$ defined as 
    \begin{equation*}
        F(x)=\frac{|\gamma^{-1}\nabla \varphi(x)|^{\frac{\gamma}{\gamma-1}}}{\varphi(x)}, \quad \text{ and } \quad  G(x)=\frac{|\gamma^{-1}\nabla \varphi(x)||x|}{\varphi(x)}
    \end{equation*}
    are equivalent, respectively, to the functions $f, \; g:\mathbb{R}_{+}^{*} \rightarrow \mathbb{R}$,
    \begin{equation*}
        f(x)=\frac{\left[\left(1+x^{2}\right)^{\frac{\gamma}{2}-1} x \right]^{\frac{\gamma}{\gamma-1}}}{\left(1+x^2\right)^{\gamma / 2}-1}, \quad \text{ and } \quad
        g(x)=\frac{\left(1+x^{2}\right)^{\frac{\gamma}{2}-1} x^{2}}{\left(1+x^2\right)^{\gamma / 2}-1}.
    \end{equation*}

     We observe that $f(x) < g(x)$. Indeed, this inequality can be expressed as
     \begin{equation*}
         \left[\left(1+x^{2}\right)^{\frac{\gamma}{2}-1} x \right]^{\frac{\gamma}{\gamma-1}} < \left(1+x^{2}\right)^{\frac{\gamma}{2}-1} x^2,
     \end{equation*}
     which simplifies to $\left(1+x^{2}\right)^{1/2}  > x$,
    a statement valid for all $x \in \mathbb{R}_{+}^{*}$. 
    Further,
    \begin{equation*}
        g^{\prime}(x)=\dfrac{2x\left(1+x^2\right)^\frac{{\gamma}-4}{2}\left(\left(1+x^2\right)^\frac{{\gamma}}{2}-\frac{\gamma}{2}x^2-1\right)}{\left(\left(1+x^2\right)^\frac{{\gamma}}{2}-1\right)^2}< 0 \quad \forall \; x>0,
    \end{equation*}
    since
    \begin{equation*}
        \begin{split}
            \left(1+x^2\right)^\frac{{\gamma}}{2}-\frac{\gamma}{2}x^2-1 
            & =\left(1+x^2\right) \left(\left(1+x^2\right)^{\frac{\gamma}{2}-1}-\frac{\gamma}{2}\right)+\left(\frac{\gamma}{2}-1\right)<0 \quad \forall \; x>0,
        \end{split} 
    \end{equation*}
    implying $g(x)\leq \lim_{x \rightarrow 0} g(x)$ for all $x \in \mathbb{R}_{+}^{*}$. Then, as
    \begin{equation*}
         \lim_{x \rightarrow 0} g(x)= \lim_{x \rightarrow 0} \left(\frac{|x|^2 }  {\left(1+|x|^{2}\right)^{\frac{\gamma}{2}}-1} \right)\overset{\underset{\text{L'H}}{}}{=}\lim_{x \rightarrow 0} \left(\frac{2x }{\gamma x\left(1+|x|^2\right)^{\frac{{\gamma}}{2}-1}} \right)= \frac{2}{\gamma}, 
    \end{equation*}
    we have $f(x)\leq \frac{2}{\gamma}$, concluding the proof.
\end{proof}
\begin{Proposition} \label{Parte-1-Proof} 
    Let $q$ and $s$ be integers such that $q$, $s>2$, $d=2$, $\alpha \in (1,2)$, $\beta$ satisfy $0\leq \beta < \alpha$, and $\gamma \in(1, \beta]$. Assume $u \in C^{\infty}(\mathbb{R}^d \times[0, \infty))$ is divergence free, and $\rho$ solves \eqref{eq-4.1-k} with $\rho_0 \geq 0 \in K_{s, \beta}$. Then the time derivative of $w(t)$, defined in \eqref{pseudo-moment}, can be expressed as
    \begin{equation}
        \label{eq-4.5-k-1}
        \frac{\mathrm{d}}{\mathrm{d} t} w(t) \leq 2\delta_{(u)} \chi^{-\mu}  w(t)+f_{(u)}(m,w), 
    \end{equation}
    where $\delta_{(u)}$ equals zero when $u=0$ and one otherwise, and 
    \begin{equation}
        \label{eq-4.5-k-2}
        f_{(u)}(m,w)=m(t) \left(C_1 +\chi^{\mu(\gamma-1)} \|u\|_{L^{\infty}}^{\gamma}-\chi  C_2 m(t)^{\frac{2}{\gamma}} \left(m(t)+2w(t)\right)^{-\frac{2-\gamma}{\gamma}}\right), 
    \end{equation}
    with $C_1$ and $C_2$ being functions of $\gamma$, and $\mu\geq 0$ a parameter to be chosen.
\end{Proposition}
\begin{proof}
    By differentiating equation \eqref{pseudo-moment}, we obtain 
    \begin{multline}
        \label{eq-4.4-k}
        \hspace{1.5cm} \frac{\mathrm{d}}{\mathrm{d} t} w=\int_{\mathbb{R}^2}\varphi(x) (u \cdot \nabla) \rho \mathrm{~d} x-\int_{\mathbb{R}^2}\varphi(x) \left(-\Delta\right)^{\alpha / 2} \rho \mathrm{~d} x+\\ \chi \int_{\mathbb{R}^2}\varphi(x) \nabla \cdot \left(\rho \nabla \Delta^{-1} \rho\right) \mathrm{~d} x-\epsilon \int_{\mathbb{R}^2}\varphi(x) \rho^q \mathrm{~d} x. \hspace{1.5cm}
     \end{multline}    
    Note that, due to \Cref{Theorem-4.1-k}, all the upcoming integrations by parts are justified for all $t \geq 0$. 
    
    As $\nabla \cdot u=0$, we have 
    \begin{equation}
        \label{eq-conta-u-1}
        \int_{\mathbb{R}^2}\varphi(x)(u \cdot \nabla) \rho \mathrm{~d} x=- \int_{\mathbb{R}^2} \nabla \varphi(x) \cdot u \rho \mathrm{~d} x.
    \end{equation}
Then, assuming that $u$ is an arbitrary smooth divergence-free vector field, and applying Young's inequality  and \Cref{Lemma-Inequality-Parte-1-Proof}, we see that
\begin{equation*}
    \begin{split}
        \left|\frac{\nabla \varphi(x)}{\gamma}\right| \left|u\right| 
        & \leq  \chi^{-\mu}\left|\frac{\nabla \varphi(x)}{\gamma}\right|^{\frac{\gamma}{\gamma-1}}+ \chi^{\mu(\gamma-1)} \left((\gamma-1)^{\gamma-1} \gamma^{-\gamma} \right)\left|u\right|^{\gamma} \\
        & \leq \frac{2}{\gamma} \chi^{-\mu} \varphi(x)+ \frac{\chi^{\mu(\gamma-1)}}{\gamma} \left|u\right|^{\gamma}, 
    \end{split}        
\end{equation*}
with $\mu>0$ to be chosen later.  
Therefore, we can rewrite \eqref{eq-conta-u-1} as 
\begin{equation*}
    \left|\int_{\mathbb{R}^2}\nabla \varphi(x) \cdot u \rho \mathrm{~d} x\right| \leq 2 \chi^{-\mu} w(t)+\chi^{\mu(\gamma-1)} \|u\|_{L^{\infty}}^{\gamma} m(t).
\end{equation*}

Next, from \Cref{Lemma-4.1}, we find that
    \begin{equation*}
        -\int_{\mathbb{R}^{2}}\left(-\Delta\right)^{\alpha / 2} \rho(x, t) \varphi(x) \mathrm{~d} x=-\int_{\mathbb{R}^{2}} \rho(x, t) \left(-\Delta\right)^{\alpha / 2}\varphi(x) \mathrm{~d} x \leq \| \left(-\Delta\right)^{\alpha / 2}\varphi \|_{L^{\infty}} \int_{\mathbb{R}^{2}} \rho(x, t) \mathrm{~d} x.
    \end{equation*}

    For the chemotaxis term, we obtain 
    \begin{equation}
        \label{chemotaxis-term}
        \begin{split}
            \int_{\mathbb{R}^2}\varphi(x) \nabla \cdot \left(\rho \nabla \Delta^{-1} \rho\right) \mathrm{~d} x &=-\int_{\mathbb{R}^2}\nabla \varphi(x) \cdot \left(\rho \nabla \Delta^{-1} \rho\right) \mathrm{~d} x\\
            & =- \int_{\mathbb{R}^2 \times \mathbb{R}^2} \nabla \varphi(x) \cdot  \frac{x-y}{|x-y|^2} \rho(x, t)  \rho(y, t) \mathrm{~d} y \mathrm{~d} x \\
            &=-\frac{1}{2} \int_{\mathbb{R}^2 \times \mathbb{R}^2} \left(\nabla \varphi(x) - \nabla \varphi(y)\right) \cdot \left(x-y \right) \frac{\rho(x, t) \rho(y, t) }{|x-y|^2} \mathrm{~d} y \mathrm{~d} x,
        \end{split}
    \end{equation}
    where in the last step we used symmetrization in $x, y$. 
    Now, observe that 
    \begin{equation*}
        \begin{split}
            m^{2} & =\int_{\mathbb{R}^2 \times \mathbb{R}^2}\rho(x, t) \rho(y, t) \mathrm{~d} x \mathrm{~d} y \\
            & =\int_{\mathbb{R}^2 \times \mathbb{R}^2}\rho(x, t) \rho(y, t) \frac{1}{\left(1+|x|^{2-\gamma}+|y|^{2-\gamma}\right)^{\gamma/2}} \left(1+|x|^{2-\gamma}+|y|^{2-\gamma}\right)^{\gamma/2} \mathrm{~d} x \mathrm{~d} y \\
            & =\int_{\mathbb{R}^2 \times \mathbb{R}^2} \left(\frac{\rho(x, t) \rho(y, t)}{1+|x|^{2-\gamma}+|y|^{2-\gamma}}\right)^{\gamma/2}\left(\rho(x, t) \rho(y, t)\right)^{\frac{2-\gamma}{2}}\left(1+|x|^{2-\gamma}+|y|^{2-\gamma}\right)^{\gamma/2} \mathrm{~d} x \mathrm{~d} y.          
        \end{split}
    \end{equation*}
    Thus, by applying the Hölder's inequality, we obtain
    \begin{equation}
         \label{conta-1}
         \begin{split}
             m^{2} & \leq \left(\int_{\mathbb{R}^2 \times \mathbb{R}^2} \frac{\rho(x, t) \rho(y, t)}{1+|x|^{2-\gamma}+|y|^{2-\gamma}} \mathrm{~d} x \mathrm{~d} y \right)^{\frac{\gamma}{2}} \\ 
            & \hspace{4cm} \left(\int_{\mathbb{R}^2 \times \mathbb{R}^2}  \rho(x, t) \rho(y, t)\left(1+|x|^{2-\gamma}+|y|^{2-\gamma}\right)^{\frac{\gamma}{2-\gamma}} \mathrm{~d} x \mathrm{~d} y\right)^{\frac{2-\gamma}{2}}.
         \end{split}        
    \end{equation}
    
    From the properties of convex functions, as $\gamma/(2-\gamma)>1$, and utilizing the inequality $|x|^{\gamma} \leq 1+\varphi(x)$, we have     
    \begin{equation*}
        \left(1+|x|^{2-\gamma}+|y|^{2-\gamma}\right)^{\frac{\gamma}{2-\gamma}} \leq C\left(1+|x|^{\gamma}+|y|^{\gamma}\right)\leq C\left(1+\varphi(x)+\varphi(y)\right),
    \end{equation*}
    for $C$ a function of $\gamma$.
    This implies that 
    \begin{equation*}
        \int_{\mathbb{R}^2 \times \mathbb{R}^2}  \rho(x, t) \rho(y, t)  \left(1+|x|^{2-\gamma}+|y|^{2-\gamma}\right)^{\frac{\gamma}{2-\gamma}} \mathrm{~d} x \mathrm{~d} y \leq C \left[m^2+2m\left(\int_{\mathbb{R}^2 }  \varphi(x) \rho(x, t) \mathrm{~d} x  \right) \right].    
    \end{equation*}

    Thus, back to \eqref{conta-1}, we have
    \begin{equation}
        \label{conta-2}
        \left(\int_{\mathbb{R}^2 \times \mathbb{R}^2} \frac{\rho(x, t) \rho(y, t)}{1+|x|^{2-\gamma}+|y|^{2-\gamma}} \mathrm{~d} x \mathrm{~d} y \right)^{\frac{\gamma}{2}} \geq C^{-\frac{2-\gamma}{2}} \frac{m^{2}}{\left(m^2+2m w(t)\right)^{\frac{2-\gamma}{2}}},
    \end{equation}
    and, considering \Cref{Lemma-4.2}, we see that
    \begin{equation*}
        \int_{\mathbb{R}^2 \times \mathbb{R}^2} \frac{\rho(x, t) \rho(y, t)}{1+|x|^{2-\gamma}+|y|^{2-\gamma}} \mathrm{~d} x \mathrm{~d} y \leq \frac{1}{K} \int_{\mathbb{R}^2 \times \mathbb{R}^2} (\nabla \varphi(x)-\nabla \varphi(y)) \cdot(x-y) \frac{\rho(x, t) \rho(y, t)}{|x-y|^{2}} \mathrm{~d} x \mathrm{~d} y.
    \end{equation*}
    Hence, from \eqref{conta-2},
    \begin{equation*}
        \int_{\mathbb{R}^2 \times \mathbb{R}^2} (\nabla \varphi(x)-\nabla \varphi(y)) \cdot(x-y) \frac{\rho(x, t) \rho(y, t)}{|x-y|^{2}} \mathrm{~d} x \mathrm{~d} y \geq K C^{-\frac{2-\gamma}{\gamma}} \frac{m^{\frac{4}{\gamma}}}{\left(m^2+2m w(t)\right)^{\frac{2-\gamma}{\gamma}}}.
    \end{equation*}
    Therefore, the chemotaxis term \eqref{chemotaxis-term} can be estimated as
    \begin{equation*}
        \begin{split}
            \int_{\mathbb{R}^2}\varphi(x) \nabla \cdot \left(\rho \nabla \Delta^{-1} \rho\right) \mathrm{~d} x & \leq -C_2 \frac{m(t)^{\frac{4}{\gamma}}}{\left(m(t)^2+2m(t)w(t)\right)^{\frac{2-\gamma}{\gamma}}} \\
            & \leq -C_2 m(t)^{1+\frac{2}{\gamma}} \left(m(t)+2w(t)\right)^{-\frac{2-\gamma}{\gamma}}
        \end{split}
    \end{equation*}
allowing us to rewrite \eqref{eq-4.4-k} as
\begin{multline*}
    \frac{\mathrm{d}}{\mathrm{d} t} w(t) \leq m(t) \left(C_1 + \chi^{\mu(\gamma-1)} \|u\|_{L^{\infty}}^{\gamma}-\chi  C_2 m(t)^{\frac{2}{\gamma}} \left(m(t)+2w(t)\right)^{-\frac{2-\gamma}{\gamma}}\right)+ 2 \chi^{-\mu} w(t)-\epsilon \int_{\mathbb{R}^2}\varphi(x) \rho^q \mathrm{~d} x,
\end{multline*}
which can be simplified to \eqref{eq-4.5-k-1}. 
\end{proof}
\begin{theorem} \label{Theorem-4.2-k}
    Let $q$ and $s$ be integers such that $q$, $s>2$, $d=2$, $\alpha \in (1,2)$, $\beta$ satisfy $0\leq \beta < \alpha$, and $\gamma \in(1, \beta]$. Suppose $u \in C^{\infty}(\mathbb{R}^d \times[0, \infty))$ is divergence free, 
    and $\rho$ solves \eqref{eq-4.1-k} with $\rho_0 \geq 0 \in K_{s, \beta}$. Then, for all $\tau>0$, we have  
    \begin{equation} 
        \label{result-0}
        \|\rho(\cdot, \tau)\|_{L^1} < \max \left\{ \psi^{\frac{\gamma}{2+\gamma}} w_0^{\frac{2}{2+\gamma}}, \; \psi  \right\},
    \end{equation}
    where 
    \begin{equation}
        \psi(\chi,u,\tau)=\frac{C_1 + \|u\|_{L^{\infty}}^{\gamma}}{\chi  C_2} \left(\frac{\Theta(\tau)}{C_1 + \|u\|_{L^{\infty}}^{\gamma}} +1 \right), \quad \quad 
        \Theta(\tau)=\begin{cases}
            \tau^{-1} & \text{ if } u=0 \\
            2\left(1-e^{- 2 \tau}\right)^{-1} & \text{ if } u \neq 0, 
        \end{cases}
    \end{equation}
   $w_0=w(0)$, and $C_1$, $C_2$ are functions of $\gamma$.   
\end{theorem}
\begin{remark} \label{Value-C-2}
    Here $C_2$ is obtained through the product of  $ 3^{-\frac{2-\gamma}{\gamma}}$ and the constant $C_2$ featured in \eqref{eq-4.5-k-2}.
\end{remark}
\begin{remark} \label{remark-Theorem-4.2-k}
    Notice that when $u=0$, if $w_0<\frac{C_1}{\chi  C_2} \left(\frac{1}{C_1\tau} +1 \right)$, the level $\|\rho(\cdot, \tau)\|_{L^1} \sim \chi^{-1}$ will be attained in at most $\tau \sim 1$.
    Otherwise, if this condition is not met, the level $\|\rho(\cdot, \tau)\|_{L^1}\sim \chi^{-1}$ will be reached in at most $\tau \sim \chi^{\frac{2}{\gamma}}$, 
    while the level $\sim \chi^{\frac{\gamma}{2+\gamma}}$ $\left( 1/3<\gamma/(2+\gamma)<1/2\right)$  in at most $\tau \sim 1$.
    Moreover, for $u\neq 0$, if $w_0<\frac{C_1}{\chi  C_2} \left(\frac{1}{C_1\tau} +1 \right)$, on the time scale $\tau \sim 1$ the level $\|\rho(\cdot, \tau)\|_{L^1}$ is $\sim \chi^{-1}$, and otherwise,  the level $\|\rho(\cdot, \tau)\|_{L^1}$ is $\sim \chi^{\frac{\gamma}{2+\gamma}}$.
\end{remark}
\begin{proof}
We begin by observing that, since $\rho_0 \in K_{s, \beta}$ and $\gamma\leq \beta$, the integral $w_0=\int_{\mathbb{R}^2} \varphi(x) \rho_0(x) \mathrm{~d} x$ is finite.
Moreover, from the ordinary differential inequality provided by \eqref{eq-4.5-k-1}, we obtain 
\begin{equation}
    \label{eq-4.6-k}
    \frac{\mathrm{d}}{\mathrm{d} t} \left(e^{- 2 \delta_{(u)} \chi^{-\mu} t} w(t) \right) \leq e^{- 2 \delta_{(u)} \chi^{-\mu} t} f_{(u)}(m,w),
\end{equation}
where $f_{(u)}$ and $\delta_{(u)}$ are defined as in  \Cref{Parte-1-Proof}. 
Proceeding, we split the proof into two cases.
\vspace{0.55cm}


$\bullet$ \underline{\textbf{Case 1 $(u=0)$:}} 
Setting $u=0$ in \eqref{eq-4.6-k} yields the differential inequality
\begin{equation}
    \label{eq-1-fractional}
    \frac{\mathrm{d}}{\mathrm{d} t} w(t) \leq f(m,w),
\end{equation}
where $f(m,w)\equiv f_{(u=0)}(m,w)$. 

Now, suppose the initial condition $(m_0, w_0)$ satisfies $f\left(m_0,w_0\right)<0$. Then,  by the continuity of $f$ in $m$ and the continuity of $m(t)$, there exists an open neighborhood $U_1$ of $t=0$ such that, for all $t  \in U_1$, 
\begin{equation*}
    \left|f\left(m(t),w_0\right)-f\left(m_0,w_0\right)\right|<-\frac{1}{2}f\left(m_0,w_0\right),
\end{equation*}
implying that $f\left(m(t),w_0\right)< f\left(m_0,w_0\right)/2$ in $U_1$. 
Furthermore, by the continuity of $f$ in both $m$ and $w$, and of $(m,w)$ with respect to $t$, there exists an open neighborhood $U_2$ of $t=0$ such that $f\left(m(t),w(t)\right)<0$ for all $t\in U_2$.  
Choosing $\tau\in U_1 \cap U_2$, we ensure $f(m(t), w(t)) <0$ for all $t \in [0, \tau]$ and $f(m(\tau), w_0)<0$.  

Since $f(m(t), w(t)) < 0$ on $[0, \tau]$, from \eqref{eq-1-fractional}, $w(t)$ decreases and, hence, $w_0 \geq w(t)$ on this interval. Moreover, as $m(t)$ is monotonically decreasing, $m(\tau) \leq m(t)$ for every $t \in [0, \tau]$. 
Consequently, $f(m(t), w(t)) \leq f(m(\tau), w_0)$ for all  $t \in [0, \tau]$. 

To clarify this assertion, consider the partial derivatives of $f(m, w)$: 
\begin{align}
    &\partial_m  f(m,w)=\frac{f(m,w)}{m} -C_2{\chi}m^\frac{2}{{\gamma}} \left(m+2w\right)^{-\frac{2}{\gamma}}\left(m+\frac{4w}{\gamma}\right), \label{derivative-f-m} \\[0.2cm]
    &\partial_w  f(m,w)=2\gamma^{-1}C_2{\chi}m^{\frac{2}{{\gamma}}+1}\left(2-{\gamma}\right)\left(m+2w\right)^{-\frac{2}{{\gamma}}}. \label{derivative-f-w}
\end{align}
Note that  $\partial_w  f(m,w)>0$ and, as $f(m,w)<0$ holds for all $t \in [0, \tau]$,  $\partial_m  f(m,w)<0$ over this interval. The signs of these derivatives reveal that $f$ decreases as  $w$ decreases, and, over the interval $[0, \tau]$, $f$ increases as $m$ decreases. Thus, the inequalities $w_0 \geq w(t)$ and $m(\tau) \leq m(t)$ for every $t \in [0, \tau]$ imply $f(m(t), w(t)) \leq f(m(\tau), w_0)$.  
In a more direct argument, we can compare $-f(m(t), w(t))$ and $-f(m(\tau), w_0)$ for $t \in [0, \tau]$ using the inequalities $-f(m(t), w(t))>0$, $w_0 \geq w(t)$ and $m(\tau) \leq m(t)$ in this interval:
\begin{equation*}
    \begin{split}
        -f\left(m(t),w(t)\right) 
        & = m(t) \left(\chi C_2 m(t) \left(1+\frac{2w(t)}{m(t)}\right)^{-\frac{2-\gamma}{\gamma}}-C_1\right) 
            \\
        & \geq m(\tau) \left(\chi C_2m(\tau)\left(1+\frac{2w_0}{m(\tau)}\right)^{-\frac{2-\gamma}{\gamma}}-C_1\right)= -f\left(m(\tau),w_0\right).  
    \end{split}
\end{equation*} 

Next, upon integrating equation \eqref{eq-1-fractional}, we derive the following estimate:
\begin{equation*}
     w_0-w(\tau) \geq -\int_0^{\tau} f\left(m(\tau),w_0\right) \mathrm{d} t= -\tau f\left(m(\tau),w_0\right).
\end{equation*}
Thus, as $w(\tau)$ cannot vanish, we must have 
$-\tau f\left(m(\tau),w_0\right) < w_0$, that is, 
\begin{equation}
    \label{eq-Theorem-4.2-01}
    \tau m(\tau) \left[\chi  C_2 m(\tau)^{\frac{2}{\gamma}} \left(m(\tau)+2w_0\right)^{-\frac{2-\gamma}{\gamma}}- C_1\right] < w_0,
\end{equation}
which can be written as 
\begin{equation*}
    m(\tau)^{1+\frac{2}{\gamma}} < \frac{C_1}{\chi  C_2}\left(\frac{w_0}{ C_1\tau}  + m(\tau)\right) \left(m(\tau)+2w_0\right)^{\frac{2-\gamma}{\gamma}}.
\end{equation*}
From this, we obtain, if $m(\tau)<w_0$, 
\begin{equation*}
    \begin{split}
        m(\tau) &< \left[\frac{C_1}{\chi  C_2}\left(\frac{1}{ C_1\tau}  + 1\right)\right]^{\frac{\gamma}{2+\gamma}}  w_0^{\frac{\gamma}{2+\gamma}}  \left(3w_0\right)^{\frac{2-\gamma}{2+\gamma}}
        =\left[\frac{C_1}{\chi  C_2} \left(\frac{1}{C_1\tau} +1 \right) \right]^{\frac{\gamma}{2+\gamma}} w_0^{\frac{2}{2+\gamma}},
    \end{split} 
\end{equation*}   
and, if $m(\tau)\geq w_0$,
\begin{equation*}
    \begin{split}
        m(\tau)^{1+\frac{2}{\gamma}} &< \left[\frac{C_1}{\chi  C_2} \left(\frac{1}{C_1\tau} +1 \right) \right] m(\tau) \left(3m(\tau)\right)^{\frac{2-\gamma}{\gamma}}=\left[\frac{C_1}{\chi  C_2} \left(\frac{1}{C_1\tau} +1 \right) \right] m(\tau)^{\frac{2}{\gamma}}, 
    \end{split}
\end{equation*}   
where we 
incorporated $3^{-\frac{2-\gamma}{\gamma}}$ into $C_2$. Therefore,
\begin{equation*}
    m(\tau) < \frac{C_1}{\chi  C_2} \left(\frac{1}{C_1\tau} +1 \right).
\end{equation*}

Note that, if the initial assumption, $f(m_0,w_0)< 0$, is not met, the result still follows from $m(t) \leq m_0$ for all $t>0$. Indeed, $f(m_0,w_0)\geq 0$ is equivalent to $ m_0^{1+\frac{2}{\gamma}} \leq \frac{C_1}{\chi  C_2}m_0 \left(m_0+2w_0\right)^{\frac{2-\gamma}{\gamma}}$. Then, as seen in the previous calculations, incorporating $3^{-\frac{2-\gamma}{\gamma}}$ into $C_2$, if $m_0 < w_0$, we deduce that $ m_0< \left(\frac{C_1}{\chi  C_2} \right)^{\frac{\gamma}{2+\gamma}} w_0^{\frac{2}{2+\gamma}}$; otherwise, 
we obtain $m_0 \leq \frac{C_1}{\chi  C_2}$. 
\vspace{0.25cm} 

$\bullet$ \underline{\textbf{Case 2 $(u\neq 0)$:}} 
Assuming the initial condition $(m_0,w_0)$ satisfies $f_{(u)}\left(m_0,w_0\right)<0$, as proved in \textbf{case 1}, there exists $\tau>0$  such that $f_{(u)}\left(m(t),w(t)\right)<0$ for all $t \in[0, \tau]$, and $f_{(u)}\left(m(\tau),w_0\right)<0$.
Additionally, $f_{(u)}(m(t), w(t)) \leq f_{(u)}(m(\tau), w_0)$ in this interval, since the partial derivatives of $f_{(u)}$ with respect to $m$ and $w$ are expressed, respectively, by equations \eqref{derivative-f-m} and \eqref{derivative-f-w}, with $C_1$ replaced by $C_1+\chi^{\mu(\gamma-1)} \|u\|_{L^{\infty}}^{\gamma}$, so that the validity of this statement is justified as previously outlined.

Subsequently, as $\delta_{(u)}=1$ in
\eqref{eq-4.6-k}, the time integral of the right-hand side of \eqref{eq-4.6-k} over the interval $[0, \tau]$ can be upper-bounded by
\begin{multline}
    \label{eq-4.7-k}
    \int_0^\tau e^{- 2 \chi^{-\mu}t} m(\tau) \left[C_1 + \chi^{\mu(\gamma-1)} \|u\|_{L^{\infty}}^{\gamma}-\chi  C_2 m(\tau)^{\frac{2}{\gamma}} \left(m(\tau)+2w_0\right)^{-\frac{2-\gamma}{\gamma}}\right] \mathrm{~d} t=\\
    \frac{\left(1-e^{- 2 \chi^{-\mu}\tau}\right)}{2} \chi^{\mu} m(\tau) \left[C_1 + \chi^{\mu(\gamma-1)} \|u\|_{L^{\infty}}^{\gamma}-\chi  C_2 m(\tau)^{\frac{2}{\gamma}} \left(m(\tau)+2w_0\right)^{-\frac{2-\gamma}{\gamma}}\right],
\end{multline}
and, to avoid the contradiction that $w(\tau)$ vanishes, we imposed that 
\begin{equation}
    \label{eq-4.8-k}
    \left(1-e^{- 2 \chi^{-\mu}\tau}\right) \chi^{\mu} m(\tau) \left[\chi  C_2 m(\tau)^{\frac{2}{\gamma}} \left(m(\tau)+2w_0\right)^{-\frac{2-\gamma}{\gamma}}-\left(C_1 + \chi^{\mu(\gamma-1)} \|u\|_{L^{\infty}}^{\gamma}\right)\right] <  2 w_0.
\end{equation}
Proceeding analogously to the proof in case 1, we can express condition \eqref{eq-4.8-k} as
\begin{equation*}
     m(\tau)^{1+\frac{2}{\gamma}} < \left(\frac{C_1 + \chi^{\mu(\gamma-1)} \|u\|_{L^{\infty}}^{\gamma}}{\chi  C_2}\right)\left(\frac{ 2 w_0}{\left(C_1 + \chi^{\mu(\gamma-1)} \|u\|_{L^{\infty}}^{\gamma}\right) \left(1-e^{- 2 \chi^{-\mu}\tau}\right) \chi^{\mu} }  + m(\tau)\right) \left(m(\tau)+2w_0\right)^{\frac{2-\gamma}{\gamma}}.
\end{equation*}
Then, incorporating, as before, $3^{-\frac{2-\gamma}{\gamma}}$ into $C_2$, we obtain, for $m(\tau)<w_0$,
\begin{equation}
    \label{eq-4.7-k-2-1}
    m(\tau) <\left[\left(\frac{C_1 + \chi^{\mu(\gamma-1)} \|u\|_{L^{\infty}}^{\gamma}}{\chi  C_2}\right)\left(\frac{ 2 }{\left(C_1 + \chi^{\mu(\gamma-1)} \|u\|_{L^{\infty}}^{\gamma}\right) \left(1-e^{- 2 \chi^{-\mu}\tau}\right) \chi^{\mu} }  + 1\right) \right]^{\frac{\gamma}{2+\gamma}} w_0^{\frac{2}{2+\gamma}},
\end{equation}   
and for $m(\tau)\geq w_0$,
\begin{equation}
    \label{eq-4.7-k-2-2}
    m(\tau) < \left(\frac{C_1 + \chi^{\mu(\gamma-1)} \|u\|_{L^{\infty}}^{\gamma}}{\chi  C_2}\right)\left(\frac{ 2 }{\left(C_1 + \chi^{\mu(\gamma-1)} \|u\|_{L^{\infty}}^{\gamma}\right) \left(1-e^{- 2 \chi^{-\mu}\tau}\right) \chi^{\mu} }  + 1\right).
\end{equation}

We can see that the optimal choice for the parameter $\mu$ (invariant to any quantity of the problem) to minimize $m(\tau)$ is zero.  
%
Finally, as in \textbf{case 1}, if the initial assumption is not met, the result follows from $m(t) \leq m_0$ for all $t>0$. 
\end{proof}

\section*{Acknowledgments}
CA was supported in part by the Coordenação de Aperfeiçoamento de Pessoal de Nível Superior - Brasil (CAPES) - Finance Code 001.
AK was partially supported by the NSF-DMS grant 2306726.


\printbibliography

\end{document}